\documentclass[reqno, 10pt]{amsart}
\usepackage{amssymb,amsmath,amsfonts,bm, mathtools,dsfont}
\usepackage{dsfont}
\usepackage{epsfig}
\usepackage{graphicx}
\usepackage{esint}
\usepackage{enumerate}
\usepackage{verbatim}
\usepackage{color}
\usepackage{caption}
\usepackage{geometry}
\usepackage{enumitem}
\usepackage{lipsum}

\usepackage{hyperref}



\def\R{{\mathbb R}}

\def\N{{\mathbb N}}
\def\L{{\mathcal L}}
\def\C{{\mathcal C}}

\def\A{{\mathcal A}}

\def\1{{\mathds{1}}}

\newcommand{\mt}{L^\infty(\R, \m)}

\newcommand{\m}{\mathcal{M}_{\alpha, \beta}}
\newcommand{\M}{L^\infty(\R, \m)}
\newcommand{\T}{T^{-(\cdot)}}
\newcommand{\U}{\mathcal{U}}

\newcommand{\lt }{\left|\left|\left|}
\newcommand{\rt}{\right|\right|\right|_{\alpha,\beta}}
\newcommand{\ld }{\left\|}
\newcommand{\rd}{\right\|_{\alpha,\beta}}

\newcommand{\rti}
{\right|\right|\right|_{\alpha,\beta}}

\date{} 

  \DeclareMathOperator*{\sign}{sgn}

\theoremstyle{plain}
\newtheorem{theorem}{Theorem}[section]
\newtheorem{definition}[theorem]{Definition}
\newtheorem{proposition}[theorem]{Proposition}

\newtheorem{lemma}[theorem]{Lemma}
\newtheorem{corollary}[theorem]{Corollary}
\newtheorem{remark}[theorem]{Remark}

\numberwithin{theorem}{section}
\numberwithin{equation}{section}
\numberwithin{figure}{section}

\let\oldtocsection=\tocsection
 
\let\oldtocsubsection=\tocsubsection
 
\let\oldtocsubsubsection=\tocsubsubsection
 
\renewcommand{\tocsection}[2]{\hspace{0em}\oldtocsection{#1}{#2}}
\renewcommand{\tocsubsection}[2]{\hspace{1em}\oldtocsubsection{#1}{#2}}
\renewcommand{\tocsubsubsection}[2]{\hspace{2em}\oldtocsubsubsection{#1}{#2}}

\begin{document}
\bibliographystyle{plain}

\parskip=4pt

\title[]
{Inhomogeneous six-wave Kinetic Equation in exponentially weighted $L^\infty$ spaces}

\author[Nata\v{s}a Pavlovi\'{c}]{Nata\v{s}a Pavlovi\'{c}}
\address{Nata\v{s}a Pavlovi\'{c},  
Department of Mathematics, The University of Texas at Austin.}
\email{natasa@math.utexas.edu}

\author[Maja Taskovi\'{c}]{Maja Taskovi\'{c}}
\address{Maja Taskovi\'{c},  
Department of Mathematics, Emory University}
\email{maja.taskovic@emory.edu}

\author[Luisa Velasco]{Luisa Velasco}
\address{Luisa Velasco,  
Department of Mathematics, The University of Texas at Austin.}
\email{lmvelasco@utexas.edu}
\begin{abstract} 
Six-wave interactions are used for modeling various physical systems, including in optical wave turbulence \cite{bola09} (where a cascade of photons displays this kind of behavior)
and in quantum wave turbulence \cite{kosv04} (for the interaction of Kelvin waves in superfluids).
In this paper, we initiate the analysis of the Cauchy problem for the spatially inhomogeneous six-wave kinetic equation. More precisely, we obtain the existence and uniqueness of non-negative mild solutions to this equation in exponentially weighted $L^\infty_{xv}$ spaces. This is accompanied by an analysis of the long-time behavior of such solutions - we prove that the solutions scatter, that is, they converge to solutions of the transport equation in the limit as $t \to \pm \infty$.
Compared with the study of four-wave kinetic equations, the main challenge we face is to address the increased complexity of the geometry of the six-wave interactions.

\end{abstract}
\maketitle
\tableofcontents

\section{Introduction}

Systems of interacting nonlinear waves are ubiquitous in nature, from quantum mechanical to astrophysical scales. However, these systems involve a large number of interacting waves, at which point it becomes useful to  consider the average behavior of the ensemble rather than that of individual waves.  In particular, one considers an effective equation, called a wave kinetic equation (WKE), that describes the asymptotic dynamics of the energy spectrum (also referred to as the wave spectrum) in the weak nonlinearity limit under certain conditions.  This 
statistical mechanics treatment of weakly nonlinear and dispersive waves is described as wave turbulence. 

One of the first appearances of a WKE can be attributed to Pierels \cite{pi29} in 1929, who formally derived a WKE to describe the evolution of the wave spectrum for phonons in anharmonic crystals. Interest in wave turbulence was renewed in the 1960s in the context of water waves \cite{bene69, besa91, ha62} and plasma physics \cite{gasa79, ve67}. Additionally, new questions arose focused on the role of wave turbulence in the transport of energy across scales, called ``energy cascades". This was led by the pivotal work of Zakharov \cite{za65} in which a new power-law type stationary solution to a WKE was discovered, analogous to the Kolmogorov spectrum of hydrodynamic turbulence. Details of these developments and physical examples can be found in the books of Zakharov, L'vov, and Falkovich \cite{zalvfa92} and Nazarenko \cite{na11}.

Since then the mathematical theory has developed significantly, with a particular focus on rigorous derivations of effective equations from dynamics governed by nonlinear dispersive equations. 

The emergence of the 4-wave spatially homogeneous kinetic equation, 
\begin{align}\label{eq:hom4}
    \partial_t f = \C_4[f] 
\end{align}
with 
\begin{align*}
\C_4[f]: =\int \delta(\Sigma)\delta(\Omega)ff_1f_2f_3\left(\frac{1}{f} + \frac{1}{f_1} - \frac{1}{f_2} - \frac{1}{f_3}\right)dv_1dv_2dv_3,\\
    \Sigma = v + v_1 - v_2 - v_3, \quad \Omega = |v|^2 + |v_1|^2 - |v_2|^2 - |v_3|^2, \quad f_i = f(t,v_i),
\end{align*}
 from the cubic non-linear Schr\"odinger  equation (NLS) has been studied by e.g. Buckmaster,  Germain,  Hani, Shatah \cite{bugehash21}, Collot, Germain  \cite{coge19, coge20} and Deng, Hani \cite{deha21,deha21_2,deha23}. We also mention a recent result of Deng, Hani \cite{deha23_2} that establishes the derivation of the homogeneous 4-wave kinetic equation from random data as long as the wave kinetic equation itself is well-posed. On the other hand, the spatially homogeneous 3-wave kinetic equation, 
 \begin{align*}
    \partial_t f = \C_3[f],
\end{align*}
(where $\C_3$ describes the effect of 3-wave interactions on the wave spectrum)
 was derived up to kinetic time \footnote{which refers to the time scale for which the kinetic behavior is expected} by Staffilani, Tran \cite{sttr21} beginning from a stochastic Zakharov-Kuznetsov equation -- a multidimensional generalization of the 1D Korteweg–De Vries equation.

There are fewer results available in the spatially inhomogeneous setting,
\begin{align}\label{eq:inhom}
\partial_tf + v\cdot \nabla f = \C_i[f],\quad i =3,4,
\end{align}
with the majority of the existing work focusing on the 3-wave kinetic equation. In particular, the spatially inhomogeneous 3-wave kinetic equation \eqref{eq:inhom} with $i=3$ was derived up to arbitrarily small polynomial loss of kinetic time by Ampatzoglou, Collot and Germain \cite{amcoge24} from Schr\"odinger type quadratic nonlinearities. Building on \cite{sttr21}, Hannani, Rosenzweig, Staffilani, and Tran \cite{harosttra22} derived the spatially inhomogeneous 3-wave kinetic equation \eqref{eq:inhom} with $i=3$   up to kinetic time from a stochastic Zakharov-Kuznetsov equation. %

While the attention of the wave turbulence community has been largely devoted to the derivation of WKE, the analysis of the WKE themselves is less developed. Germain, Ionescu, Tran \cite{geiotr20} proved local well-posedness of the spatially homogeneous 4-wave kinetic equation \eqref{eq:hom4} in polynomially weighted $L^2_v$ and $L^\infty_v$ spaces. In the case of the spatially inhomogeneous WKE, Ampatzoglou \cite{am20} proved small data global well-posedness of the 4-wave kinetic equation \eqref{eq:inhom} with $i=4$ in exponentially weighted $L_{x,v}^\infty$ spaces. The well-posedness theory for the same equation was then expanded by Ampatzoglou, Miller, and the first two authors of this paper in the work \cite{ammipata24} which proved well-posedness of the inhomogeneous 4-wave kinetic equation and its hierarchy in polynomially weighted $L^\infty_{x,v}$ spaces. Most recently, Ampatzoglou, L\'{e}ger \cite{amle24} continued the analysis of this equation using dispersive techniques to prove an existence result in $L^1_{x,v}$ intersected with $L^\infty_{x,v}$ spaces polynomially weighted in velocity, and to study the asymptotic behavior of solutions.

The aim of this paper is to study a Cauchy problem for the 6-wave kinetic equation (6-WKE), which, 
for a function $f: \R\times \R^d \times \R^d \to \R$ and initial data $f_0: \R^d \times \R^d \to \R$, is given by 
\begin{equation}\label{eq:6wave}\begin{cases}
    \partial_t f + v\cdot \nabla_x f = \C_6[f],\\
    f(t = 0) = f_0.
\end{cases}  
\end{equation}
Here the collision operator is defined by
\begin{align}\label{collision_op}
    \C_6[f](t,x,v) = \int_{\R^{5d}}\delta(\Sigma)\delta(\Omega)ff_1f_2f_3f_4f_5\left(\frac{1}{f} + \frac{1}{f_1} + \frac{1}{f_2} - \frac{1}{f_3} - \frac{1}{f_4} - \frac{1}{f_5}\right)dv_1dv_2dv_3dv_4dv_5,
\end{align}
where the resonant manifolds are given by  
\begin{align}\label{res_manifolds}
    \Sigma = v + v_1 + v_2 - v_3 - v_4 - v_5, \quad \Omega = |v|^2 + |v_1|^2 + |v_2|^2 - |v_3|^2 - |v_4|^2 - |v_5|^2
\end{align}
and we use the notation $f = f(t,x,v)$ and $f_i = f(t,x,v_i)$ for $i = 1,2,3,4,5$.

 Six-wave interactions are relevant for the task of modeling various physical systems. Indeed, in  \cite{bola09} Bortolozzo,  Laurie,  Nazarenko report the first experimental evidence of optical wave turbulence, observing a cascade of photons from a system characterized by six-wave interactions. Moreover, as described in \cite{bola09}, this behavior is predicted by the one dimensional optical wave turbulence theory, as one dimensional systems whose governing equations correspond to dispersion relation $\omega = k^2$ can be shown to lack 4-wave resonant interactions, allowing the next order quintic terms to control the dynamics. Additionally, six-wave interactions find application in modeling quantum wave turbulence, specifically the interaction of Kelvin waves in superfluid, see Kozik, Svistunov \cite{kosv04}.

With this physical motivation, Banks, Buckmaster, Korotkevich, Kova\v{c}i\v{c}, Shatah \cite{babu22} investigated the parameter regimes for which the average dynamics of a wave ensemble is accurately modeled by the homogeneous 6-WKE for a system of finite size via formal derivation from the defocusing quintic NLS and numerical analysis.  Additionally, the work of de Suzzoni \cite{su22} makes progress towards a rigorous derivation of the homogeneous 6-wave kinetic equation. 

However, to the best of our knowledge, there has not been any work on existence, uniqueness, and long-time behavior of solutions to the Cauchy problem for the 6-WKE \eqref{eq:6wave}. This paper aims to fill in this gap. More precisely, we prove global existence and uniqueness of solutions to \eqref{eq:6wave} when $d = 1$ in \textit{exponentially weighted $L^\infty$ spaces} (see \eqref{spacetime} for the precise definition of this space and Theorems \ref{thm:wke well posed} - \ref{thm:KSnonnegative} for the well-posedness result). Moreover, we remark that Theorems \ref{thm:nonnegativity} - \ref{thm:KSnonnegative} yield the existence of nonnegative solutions which are physically admissible \footnote{The wave kinetic equation governs the evolution of the energy spectrum -- given by the expectation of the square of the absolute value of the Fourier mode  -- and so we expect physical solutions to be non-negative.}.  Also, inspired by results on long-time behavior of solutions to dispersive equations (see e.g. Cazenave\cite{Cazenave}) and the Boltzmann equation (see Bardos, Gamba, Golse, Levermore \cite{baga16}),
we extend our analysis to include the asymptotic behavior of solutions. 
In particular, we prove that solutions scatter, that is, they converge to solutions of the transport equation in the limit as $t \to \pm \infty$ (see Theorems \ref{thm: scattering 1} and \ref{thm:scattering 2}).

 Due to the higher order of the nonlinearity as compared to the 3 and 4-wave kinetic equations, there is increased complexity in the geometry of the wave interactions encoded in the resonant manifolds \eqref{res_manifolds}. Addressing this key difficulty informed our choice of the solution space and required careful attention in the formulation of crucial a priori bounds used in proofs of the above results. In particular, we use a parametrization inspired by the one for three-particle collisions present in the ternary Boltzmann equation \cite{ampa21}.

\subsection*{Organization of the paper.}
In Section \ref{sec:main results} we define the spaces used in this paper, as well as the notion of the solution, and we state our main results. The key representation lemma (Lemma \ref{lem:I}) and a priori estimates on the collision operator are presented in Section \ref{sec:prep}, while existence and uniqueness of solutions is obtained via a fixed point argument in Section \ref{sec:GWP}. In addition, in Section \ref{sec:GWP}, by applying a fixed point argument centered around a Maxwellian we obtain existence and uniqueness of non-negative solutions.
Section \ref{sec:KS} gives another proof of existence and uniqueness of non-negative solutions (under fewer parameter constraints  than the argument provided in Section \ref{sec:GWP}) using Kaniel-Shinbrot method. The long-time behavior of our solutions is identified in Section \ref{sec:scattering}.
 
\subsection*{Acknowledgements.}  N.P. gratefully acknowledges support from the NSF under grants No. DMS-1840314, DMS-2009549 and DMS-2052789. M.T. gratefully acknowledges support from the NSF grant DMS-2206187.  L.V. gratefully acknowledges support from the NSF via Graduate Research Fellowship Program. We also thank Simons Collaboration on Wave Turbulence for organizing annual meetings and summer schools that inspired our interest in this topic.

\section{Function spaces and main results} \label{sec:main results}
In this section, we define our working function spaces and present the statements of our main theorems. 

\subsection{Function Spaces}
We now introduce the function spaces that will be used throughout this paper. First, we define sets of space-velocity functions, 
\begin{align*}
    F_{x,v}&: = \{f: \R^d \times \R^d \to \R: f \text{ is measurable}\},\\
    F_{x,v}^+&: = \{f: f \in F_{x,v}: f(x,v) \geq 0\},\\
    L^{1,+}_{x,v}&: = L^1_{x,v}\cap F_{x,v}^+,
\end{align*}
where $L^1_{x,v}$ denotes the Lebesgue space in $x$ and $v$.
Now, 
let us define sets of time dependent functions
\begin{align*}
    \mathcal{F} &:= \{f:\R \to F_{x,v}\},\\
    \mathcal{F}^+&:= \{f: \R \to F^+_{x,v}\}.
\end{align*}
Finally, we define the spaces we will work with - exponentially weighted $L^\infty$ spaces - and that will be used in the well-posedness and scattering results for \eqref{eq:6wave}. 
\begin{definition}
    Given $\alpha,\beta > 0$, we define the space 
    \begin{align}\label{space}
    \m := \left\{f \in F_{x,v}: \ld f\rd < \infty\right\},
    \end{align}
    where we define the norm as
    \begin{align} 
    \ld f\rd: = \sup_{x,v}|f(x,v)|e^{\alpha|x|^2 + \beta|v|^2}.
    \end{align}
    Further, we define non-negative functions bounded by a Maxwellian as follows 
    \begin{align}
        \m^+ = \m \cap F_{x,v}^+.
    \end{align}
     Additionally, 
     we define a family of function spaces in time as 
     \begin{align}\label{spacetime}
         \mt : = \{f \in \mathcal{F}:\lt f \rt < \infty\},
     \end{align} where we use the following norm 
     \begin{align}
         \lt f \rt: = \sup_{ t \in \R}\ld f(t)\rd.
     \end{align}
     We also define 
    \begin{align}
        \M^+ = \M \cap \mathcal{F}^+.
    \end{align}
\end{definition}

Before we define a notion of solution, we introduce the transport operator $(T^s)_{s \in \R}$ acting on function $g:\R \times \R^d \times \R^d \rightarrow \R$ as follows: 
\begin{align}\label{eq:transport}
    T^sg(t,x,v): = g(t,x - sv, v).
\end{align}
Throughout the paper we  will also be using an abbreviated notation
\begin{align}\label{abbr}
    T^{-(\cdot)}g := T^{-t} g(t,x,v).
\end{align}

With this notation, we can formally rewrite the wave kinetic equation \eqref{eq:6wave} via Duhamel's formula as follows: 
\begin{align}
    f(t) = T^{t}f_0 + \int_0^t T^{t-s}\C[f](s)\ ds,\quad t \in \R,
\end{align}
or equivalently
\begin{align}\label{T-Duhamel}
    T^{-t}f(t) = f_0 + \int_0^t T^{-s}\C[f](s)\ ds,\quad t \in \R.
\end{align}
This last expression \eqref{T-Duhamel} is used to define a notion of a mild solution. This type of solution is used in e.g. \cite{am20, amga22}.

\begin{definition}\label{def:solution}
    Let $\alpha,\beta > 0$.
    We consider initial data $f_0 \in \m$. A measurable function $f:\R\times \R^d\times \R^d \to \R$ is called a mild solution of (\ref{eq:6wave}) corresponding to initial data $f_0$ if 
 $
        T^{-t}f(t) \in  L^\infty(\R, \m),   
$    
and 
\begin{align}
    T^{-t}f(t,x,v) = f_0(x,v) + \int_0^t T^{-s}\C[f](s,x,v)\ ds,\quad t \in \R.
\end{align}
\end{definition}
\subsection{Main results}

We start by stating our results for the 6-wave kinetic equation on the existence of global solutions for small data as well as the existence of non-negative solutions -- this is the content of the next two theorems. 

\begin{theorem}[Global well-posedness]\label{thm:wke well posed}
    Let $\alpha,\beta > 0$, $d = 1$ and $0 < R_e \leq \frac{\alpha^{1/8}}{2^{\frac{7}{2}}C_{1,\beta}^{1/4}}$, where $C_{1,\beta} > 0$ is given by \eqref{conv_constant}.
    \begin{enumerate}[label= (\alph*)]
        \item  Let $f_0 \in \m$ with $\ld f_0\rd \leq R_e$. Then equation (\ref{eq:6wave}) has a unique mild solution $f$ (in the sense of definition \ref{def:solution}) satisfying the bound 
    \begin{align}\label{sol bound}
    \lt T^{-t}f(t)\rt \leq 2R_e. 
    \end{align}
        \item Moreover, if $f_0,g_0 \in \m$ with $\ld f_0\rd,\ld g_0\rd \leq R_e$, and $f$ and $g$ are the corresponding mild solutions to (\ref{eq:6wave}) satisfying the bound \eqref{sol bound} , then following stability estimate holds: 
    \begin{align}\label{stability}
    \lt T^{-t}f(t) - T^{-t}g(t)\rt \leq 2\ld f_0 - g_0\rd.
    \end{align}
    \end{enumerate}
where $T^{-t}$ is defined in \eqref{eq:transport}.
\end{theorem}

\begin{theorem}[Existence and uniqueness of non-negative solutions]\label{thm:nonnegativity}
Let 
$\alpha,\beta > 0$, $d = 1$, and let
$C_{1,\beta}$ be given by \eqref{conv_constant}.
Furthermore, assume that 
\begin{align}\label{condition 1}
    C_{1,\beta}\alpha^{-1/2} < \left(\frac{3}{16}\right)^4,
\end{align}
and let $R_p>0$ be such that
\begin{align}\label{condition 2}
  \frac{1}{6} \leq R_p \leq   \frac{\alpha^{1/8}}{2^3\sqrt[4]{ C_{1,\beta}}} - \frac{1}{2}.
\end{align}
Let $M$ be the Gaussian given by 
\begin{align}
    M(x,v) = e^{-\alpha |x|^2 - \beta |v|^2}.
\end{align}
Then we have the following
\begin{enumerate}[label = (\alph*)]
    \item If $f_0 \in \m$ is such that
\begin{align}
    \ld f_0 - M\rd < R_p,
\end{align}
then there exists a unique mild solution $f$ of the 6-wave kinetic equation \eqref{eq:6wave} corresponding to the initial data $f_0$ satisfying the bound
\begin{align}
    \lt T^{-t}f(t) - M\rt \leq 2R_p.
\end{align} 
\item Moreover, if $\frac{1}{2^{12}} < C_{1,\beta} \alpha^{-1/2} < \left(\frac{3}{16}\right)^4$, then the mild solution $f$ is non-negative.

\end{enumerate}
\end{theorem}
  
The proofs of Theorem \ref{thm:wke well posed} and Theorem \ref{thm:nonnegativity} are based on a fixed point argument with the additional restrictions on $\alpha,\beta$ in Theorem \ref{thm:nonnegativity} providing sufficient conditions for the existence of non-negative solutions. However, we can obtain both global well-posedness and non-negativity results without these additional restrictions on $\alpha,\beta$ by employing the Kaniel-Shinbrot iteration method that was used previously  for establishing well-posedness results  for various kinetic equations in \cite{kash78, ilsh84, beto85, to88, beto87, pato89, al09, alga09,st10, amga22}. This is the content of the following theorem.

\begin{theorem}[Kaniel-Shinbrot] \label{thm:KSnonnegative}
    Let $\alpha,\beta > 0$, $d = 1$,
    and $0 < R_{ks} \leq \frac{19}{20}\left(\frac{\alpha^{1/8}}{\sqrt[4]{240C_{1,\beta}}}\right)$ where $C_{1,\beta} > 0$ is given by \eqref{conv_constant}. If $f_0 \in \m^+$ with $||f_0||_{\alpha,\beta} \leq R_{ks}$, then equation \eqref{eq:6wave} has a unique non-negative mild solution (in the sense of definition \ref{def:solution}).
\end{theorem}

We are also interested in describing the long-time behavior of solutions of the 6-wave kinetic equation \eqref{eq:6wave} by comparing them with those of the linear problem. To do this, we start by defining the concept of a scattering state.
\begin{definition}
    Given a solution $f$ of the 6-wave kinetic equation \eqref{eq:6wave},  we say that $f_{\pm} \in \m$ is a scattering state if 
     \[\lim\limits_{t\to\pm\infty}\ld T^{-t}f(t) - f_\pm\rd = 0.\]
\end{definition}
Before we state the theorems, we define the following sets for fixed $R > 0$, 
\begin{align}
    B(R) &=\{g \in \mathcal{M}_{\alpha,\beta}: \| g\|_{\alpha,\beta} < R\},\label{B(R)}\\
    \overline{B}(R) &= \{g \in \overline{\mathcal{M}}_{\alpha,\beta}: \lt g\rt < R\}.\label{oB(R)}
\end{align}
The next theorem shows that every function in a small enough ball is a scattering state of a solution of \eqref{eq:6wave}.
\begin{theorem}\label{thm: scattering 1}
    Let $d=1, \alpha, \beta >0$, and $R_s \leq \frac{\alpha^{1/8}}{2^{\frac{7}{2}}C_{1,\beta}^{1/4}}$,  where $C_{1,\beta} > 0$ is given by \eqref{conv_constant}. 
    Then for every $f_\pm \in B(R_s)$, there exists a unique $f_0 \in B(2R_s)$ such that the mild solution $f$ of \eqref{eq:6wave} corresponding to the initial data $f_0$ satisfies \begin{align}\label{scatter1_lim}
    \lim\limits_{t\rightarrow \pm\infty}\ld T^{-t} f(t) - f_\pm\rd = 0.
    \end{align}
\end{theorem}
Conversely, in the following theorem, we show that for small enough initial data there exists a corresponding scattering state.
\begin{theorem}\label{thm:scattering 2}
     Let $d=1, \alpha,\beta > 0$, and $R_s \leq \frac{\alpha^{1/8}}{2^{\frac{7}{2}}C_{1,\beta}^{1/4}}$,  where $C_{1,\beta} > 0$ is given by \eqref{conv_constant}.  Let $f_0 \in B(R_s)$ and let $f$ be the mild solution to wave kinetic equation \eqref{eq:6wave} corresponding to $f_0$. Then there exists a unique $f_\pm \in B(2R_s)$ such that \[\lim\limits_{t\to\pm\infty}\ld T^{-t}f(t) - f_\pm\rd = 0.\]
\end{theorem}
Our next goal is to connect scattering states $f_-$ and $f_+$ when possible. 
\begin{theorem}\label{thm: main scattering}
    Let $d = 1,\alpha,\beta> 0$ and $0 < R_s \le \frac{\alpha^{1/8}}{2^{\frac{7}{2}}C_{1,\beta}^{1/4}}$. Then:
    \begin{itemize}
    
    \item[(i)] For each $f_{-} \in B(\frac{R_s}{2})$, there exists a unique $f_{+} \in B(2R_s) $ and a unique mild solution $f$ of \eqref{eq:6wave} such that \begin{align*}
        \lim\limits_{t \to  \infty}\ld T^{-t}f(t) - f_+\rd = 0  \quad \text{and} \quad \lim\limits_{t \to -\infty}\ld T^{-t}f(t) - f_{-}\rd = 0.
    \end{align*}

    \item[(ii)] For each $g_{+} \in B(\frac{R_s}{2})$, there exists a unique $g_{-} \in B(2R_s) $ and a unique mild solution $g$ of \eqref{eq:6wave} such that \begin{align*}
        \lim\limits_{t \to  \infty}\ld T^{-t}g(t) - g_+\rd = 0  \quad \text{and} \quad \lim\limits_{t \to -\infty}\ld T^{-t}g(t) - g_{-}\rd = 0.
    \end{align*}
\end{itemize}
\end{theorem}

\begin{remark}
    The part (i) of the above theorem implies the existence of an operator $\mathcal{S}$, defined on a subset of $\m$, that maps $f_-$ to $f_+$ i.e. $\mathcal{S}f_- = f_+.$ Such an operator is typically called the scattering operator.  The part (ii) of the above theorem implies that the operator $\mathcal{S}$ is onto on an appropriate set.
\end{remark}

\section{Preparations}\label{sec:prep}
The goal of this section is to present the tools  used in the proofs of the main theorems stated above. To do this, we begin with a crucial lemma in Subsection \ref{subsec:rep_lemma} that provides a alternative representation of the collision kernel. Then, we decompose the collision operator into  components in Subsection \ref{subsec:gain_loss} and establish several apriori bounds for these components in Subsection \ref{subsec:a priori}. Lastly, in Subsection \ref{subsec:prep} we define a key mapping that will be used to streamline the proofs of Theorems \ref{thm:wke well posed} and \ref{thm:nonnegativity} in Section \ref{sec:GWP}
 and Theorems \ref{thm: scattering 1} and \ref{thm:scattering 2} in Section  \ref{sec:scattering}.

\subsection{Representation Lemma}\label{subsec:rep_lemma}
 Inspired by the representation lemma in \cite{am24} and the parametrization used in the derivation of the ternary Boltzmann equation \cite{ampa21, am20}, we prove the following representation lemma relevant for the 6-wave kinetic equation. 

\begin{lemma}\label{lem:I}
Let $v,v_1,v_2 \in \R^d$, and denote 
    \[I(v,v_1,v_2) =\int_{\R^{3d}}\delta(\Sigma)\delta(\Omega)dv_3dv_4dv_5,\]
where 
$\Sigma = v + v_1 + v_2 - v_3 - v_4 - v_5$ and $\Omega = |v|^2 + |v_1|^2 + |v_2|^2 - |v_3|^2 - |v_4|^2 - |v_5|^2$.
Then 
\begin{align}\label{rep}
        I(v,v_1,v_2) = \int_{\mathbb{S}^{2d-1}} 
        \hspace{-6pt}
        \frac{\big(\omega_1\cdot(v_1 - v) + \omega_2\cdot(v_2 - v)\big)^{2d-2}}{(1 + \omega_1\cdot\omega_2)^{2d-1}}
\cdot\frac{\sign\big(\omega_1\cdot(v_1 - v) + \omega_2\cdot(v_2 - v)\big)+1}{4}
         d\boldsymbol{\omega},
    \end{align}
where $\boldsymbol{\omega} = (\omega_1,\omega_2) \in \mathbb{S}^{2d-1}$.
As a consequence, we have
\begin{align}\label{rep-ineq}
        I(v,v_1,v_2) \le \frac{1}{2}\int_{\mathbb{S}^{2d-1}} 
        \hspace{-6pt}
        \frac{\big|\omega_1\cdot(v_1 - v) + \omega_2\cdot(v_2 - v)\big|^{2d-2}}{(1 + \omega_1\cdot\omega_2)^{2d-1}}~ d\boldsymbol{\omega}.
\end{align}
\end{lemma}

\begin{proof}
Let $v_5 = v + v_1 + v_2 - v_3 - v_4$ as given by the resonant manifold $\Sigma$. Then we can write
    \begin{align*}
        I(v,v_1,v_2) 
        & = \int_{\R^{2d}}\delta(|v|^2 + |v_1|^2 + |v_2|^2 - |v_3|^2 - |v_4|^2 - |v + v_1 + v_2 - v_3 - v_4|^2)\ dv_3dv_4.
    \end{align*}
We make the spherical change of variable $(v_3,v_4) \in \R^d \times \R^d$ to $(c, \boldsymbol{\omega}) \in [0,\infty)\times \mathbb{S}^{2d-1}$ given by 
\begin{align}\label{COV}
    \begin{pmatrix}
            v_3\\
            v_4
        \end{pmatrix} = \begin{pmatrix}v_1\\v_2\end{pmatrix} - c\begin{pmatrix}\omega_1\\\omega_2\end{pmatrix}.
\end{align}
This change of variable is inspired by the constraint used in the definition of the collisional law in the derivation of the ternary Boltzmann equation (see (1.9) in \cite{ampa21}). 
Then we have
\begin{align}
        I&(v,v_1,v_2) \nonumber\\
        & = \int_{\mathbb{S}^{2d-1}}\int_0^\infty c^{2d - 1}\delta\left(|v|^2 + |v_1|^2 + |v_2|^2 - |v_1 - c\omega_1|^2 - |v_2 - c\omega_2|^2 - |v + c(\omega_1 + \omega_2)|^2\right)\ dcd\boldsymbol{\omega}\nonumber\\
        &= \int_{\mathbb{S}^{2d-1}}\int_0^\infty c^{2d - 1}\delta\Big(2c\,\omega_1\cdot v_1 + 2c\, \omega_2\cdot v_2 - 2c\,(\omega_1 + \omega_2)\cdot v - c^2(|\omega_1|^2 + |\omega_2|^2 + |\omega_1 + \omega_2|^2)\Big)dc d\boldsymbol{\omega}\nonumber\\
        & = \int_{\mathbb{S}^{2d-1}}J(\boldsymbol{\omega},v,v_1,v_2)d\boldsymbol{\omega},\label{IJ int}
    \end{align}
where 
\begin{align}
    J(\boldsymbol{\omega},v,v_1,v_2): = \int_0^\infty c^{2d-1}\delta\left(2c\bigg[\omega_1\cdot(v_1 - v) + \omega_2\cdot(v_2 - v) - c(1 + \omega_1 \cdot \omega_2)\bigg]\right)dc.\label{def J}
\end{align}
Using the property  $\delta(\alpha x ) = \frac{1}{|\alpha|} \delta(x)$ and the fact that $1+\omega_1 \cdot \omega_2 \ge 1/2$, we have
\begin{align*}
     J(\boldsymbol{\omega},v,v_1,v_2)
     & = \int_0^\infty \frac{c^{2d-1}}{2(1+\omega_1 \cdot \omega_2)} \, \delta\left( c^2 - c~ \frac{\omega_1\cdot(v_1 - v) + \omega_2\cdot(v_2 - v)}{1+\omega_1 \cdot \omega_2}\right)~dc\\
      & =
    \frac{1}{2(1+\omega_1 \cdot \omega_2)} \int_0^\infty c^{2d-1} ~ \delta(c^2 - Ac) ~ dc,
\end{align*}
where we use notation
\begin{align*}
    A = \frac{\omega_1\cdot(v_1 - v) + \omega_2\cdot(v_2 - v)}{1+\omega_1 \cdot \omega_2}.
\end{align*}
Next we use  that for a differentiable function $f(x)$  with simple roots $\{x_n\}_n$ we have \cite[Sec II.2.5]{gesh64}:
\begin{align*}
    \delta(f(x)) = \sum_n \frac{\delta(x-x_n)}{|f'(x_n)|}.
\end{align*}
Thanks to this property, we have
\begin{align*}
 J(\boldsymbol{\omega},v,v_1,v_2)
    & =
     \frac{1}{2(1+\omega_1 \cdot \omega_2)} \int_0^\infty c^{2d-1}
    \left( \frac{\delta(c) + \delta(c-A)}{|A|}\right) ~ dc\\
    & =
    \begin{cases}
        0, & \mbox{if} ~ A\le 0\\
        \displaystyle \frac{A^{2d-2}}{2(1+\omega_1 \cdot \omega_2)},  & \mbox{if} ~ A> 0
    \end{cases}\\
    & =
    \frac{1}{2} \frac{\Big(\omega_1\cdot(v_1 - v) + \omega_2\cdot(v_2 - v)\Big)^{2d-2}}{\big(1+\omega_1 \cdot \omega_2\big)^{2d-1}} 
 \cdot\frac{\sign\big(\omega_1\cdot(v_1 - v) + \omega_2\cdot(v_2 - v)\big)+1}{2}
\end{align*}
Then, by \eqref{IJ int}, the representation\eqref{rep} follows. Inequality \eqref{rep-ineq} is a simple consequence of the fact that $\sign\big(\omega_1\cdot(v_1 - v) + \omega_2\cdot(v_2 - v)\big)+1 \le 2$.
\end{proof}

\subsection{Decomposition of the collision operator}\label{subsec:gain_loss}
For many calculations it will be useful to split the collision operator $\C$ given in \eqref{collision_op} into the following terms: 
\begin{align}\label{gain-loss split}
    \C = G - L,
\end{align} with
\begin{align}
    &G[f,g,h,k,l,m](t,x,v) = \int_{\R^{5d}}\delta(\Sigma)\delta(\Omega)k_3l_4m_5(fg_1 + fh_2 + g_1h_2)dv_1dv_2dv_3dv_4dv_5, \label{def-g}\\
    &L[f,g,h,k,l,m](t,x,v) = \int_{\R^{5d}}\delta(\Sigma)\delta(\Omega)fg_1h_2(k_3l_4 + k_3m_5 + l_4m_5)dv_1dv_2dv_3dv_4dv_5. \label{def-l}
 \end{align}
We note that sometimes we write 
\begin{align}
L[f,g,h,k,l,m] = f \, R[g,h,k,l,m],
\end{align}
where 
\begin{align} \label{def-R}
     R[g,h,k,l,m] = \int_{\R^{5d}}\delta(\Sigma)\delta(\Omega)g_1h_2(k_3l_4 + k_3m_5 + l_4m_5)dv_1dv_2dv_3dv_4dv_5.
\end{align}

We also note that for $f,g \geq 0$ such that $f \leq g$, we have
\begin{align}\label{GL_monotonic}
    G[f,f,f,f,f,f]\leq G[g,g,g,g,g,g] \quad \text{and}\quad L[f,f,f,f,f,f]\leq L[g,g,g,g,g,g].
\end{align}
That is, the operators $G$ and $L$ applied to non-negative functions are monotone.

We can  further split $G$ and $L$ terms as follows
\begin{align}\label{G and L split}
    G=G_1 +G_2,
    \quad 
    L = L_1 + L_2,
\end{align} 
where   
\begin{align}
G_1[f,g,h,k,l,m](t,x,v) &= \int_{\R^{5d}}\delta(\Sigma)\, \delta(\Omega) \, k_3 \, l_4 \, m_5 \, f\, (g_1 + h_2) \,dv_1dv_2dv_3dv_4dv_5,\\
G_2[g,h,k,l,m](t,x,v) &= \int_{\R^{5d}}\delta(\Sigma)\,\delta(\Omega)\, k_3\, l_4\, m_5\, g_1 \, h_2 \, dv_1dv_2dv_3dv_4dv_5,\\
L_1[f,g,h,k,l,m](t,x,v) & = \int_{\R^{5d}}\delta(\Sigma)\, \delta(\Omega)\,f\, g_1\, h_2\, k_3\, (l_4 + m_5)\, dv_1dv_2dv_3dv_4dv_5,\\
L_2[f,g,h,l,m](t,x,v) & = \int_{\R^{5d}}\delta(\Sigma)\, \delta(\Omega)\, f\, g_1\, h_2\, l_4\, m_5\, dv_1dv_2dv_3dv_4dv_5.
\end{align}
Note that $G_1$ is linear in $k,l,m,f$ and the vector $(g,h)$, while $G_2$ is linear in $g,h,k,l,m$. Similarly, $L_1$ is linear in $f,g,h,k$ and the vector $(l,m)$, while $L_2$ is linear in $f,g,h,l,m$. Therefore, we have the following  decompositions:
\begin{align}
    G_1&[f,g,h,k,l,m] - G_1[\widetilde{f},\widetilde{g},\widetilde{h},\widetilde{k},\widetilde{l},\widetilde{m}]\nonumber\\
    &= G_1[f - \widetilde{f},g,h,k,l,m] + G_1[\widetilde{f},g,h,k - \widetilde{k},l,m] + G_1[\widetilde{f},g,h,\widetilde{k},l - \widetilde{l}, m] \nonumber\\
    &\qquad + G_1[\widetilde{f},g,h,\widetilde{k}, \widetilde{l}, m - \widetilde{m}] + G_1[\widetilde{f},g - \widetilde{g},h - \widetilde{h},\widetilde{k},\widetilde{l},\widetilde{m}],\label{G1 decomp}\\
    G_2&[g,h,k,l,m] - G_2[\widetilde{g},\widetilde{h},\widetilde{k},\widetilde{l},\widetilde{m}]\nonumber\\
    & = G_2[g -\widetilde{g},h,k,l,m] + G_2[\widetilde{g},h - \widetilde{h}, k, l, m] + G_2[\widetilde{g},\widetilde{h},k - \widetilde{k},l,m]\nonumber\\
    &\qquad + G_2[\widetilde{g},\widetilde{h},\widetilde{k}, l - \widetilde{l}, m] + G_2[\widetilde{g},\widetilde{h},\widetilde{k},\widetilde{l}, m - \widetilde{m}],\label{G2 decomp}
\end{align}
\begin{align}
    L_1&[f,g,h,k,l,m] - L_1[\widetilde{f},\widetilde{g},\widetilde{h},\widetilde{k},\widetilde{l},\widetilde{m}]\nonumber\\
    & = L_1[f - \widetilde{f},g,h,k,l,m] + L_1[\widetilde{f},g - \widetilde{g},h,k,l,m] + L_1[\widetilde{f},\widetilde{g},h - \widetilde{h},k,l,m]\nonumber\\ 
    &+ L_1[\widetilde{f},\widetilde{g},\widetilde{h},k - \widetilde{k},l,m] + L_1[\widetilde{f},\widetilde{g},\widetilde{h},\widetilde{k}, l - \widetilde{l}, m - \widetilde{m}],\label{L1 decomp}\\
    L_2&[f,g,h,l,m] - L_2[\widetilde{f},\widetilde{g},\widetilde{h},\widetilde{l},\widetilde{m}]\nonumber\\
    & = L_2[f - \widetilde{f},g,h,l,m] + L_2[\widetilde{f},g - \widetilde{g}, h, l,m] + L_2[\widetilde{},\widetilde{g},h - \widetilde{h}, l,m]\nonumber\\
    & \qquad + L_2[\widetilde{f},\widetilde{g},\widetilde{h},l - \widetilde{l}, m]+ L_2[\widetilde{f},\widetilde{g},\widetilde{h},\widetilde{l},m - \widetilde{m}].\label{L2 decomp}
\end{align}
The proof of these linear decompositions can be found in Appendix \ref{lin_decomp}.

\subsection{A priori estimates}\label{subsec:a priori}
We present several estimates of the transported $G$ and $L$ operators that will be used in the proofs of the main theorems. 

We begin with a  lemma that will help us simplify the subsequent a priori bounds.

\begin{lemma}\label{lem:step1}
Let $\alpha,\beta > 0, d \geq 1$ and suppose that functions $f,g,h,k,l,m \in \mathcal{F}$ are such that $T^{-(\cdot)}f$, \, $T^{-(\cdot)}g$, $T^{-(\cdot)}h, \, T^{-(\cdot)}k, \, T^{-(\cdot)}l, \, T^{-(\cdot)}m \in \M$. Then for any $s \in \R$, if we denote
\begin{align} \label{def:I}
    \Gamma(s,x,v) = e^{-\alpha|x|^2 - \beta|v|^2}  \int_{\R^{2d}}(|v_1 - v| + |v_2 - v|)^{2d-2} e^{-\alpha(|x + s(v - v_1)|^2 + |x + s(v - v_2)|^2) - \beta(|v_1|^2 + |v_2|^2)}dv_1dv_2,
\end{align}
then we have
    \begin{align}
        & \bigg| T^{-s} L_1[f,g,h,k,l,m](s) \bigg| \le 2^{2d-1} \lt T^{-(\cdot)}f\rt\ \lt T^{-(\cdot)}g\rt\ \lt T^{-(\cdot)}h\rt\ \lt T^{-(\cdot)}k\rt \nonumber\\
        & \hspace{4.3cm}\times \bigg(\lt T^{-(\cdot)}l\rt + \lt T^{-(\cdot)}m\rt\bigg) 
          \Gamma(s,x,v), \label{L1_step1}\\
        & \bigg| T^{-s} G_1[f,g,h,k,l,m](s) \bigg| \le 2^{2d-1} \lt T^{-(\cdot)}f\rt\ \lt T^{-(\cdot)}k\rt\ \lt T^{-(\cdot)}l\rt\ \lt T^{-(\cdot)}m\rt \nonumber\\
        & \hspace{4.3cm}\times \bigg(\lt T^{-(\cdot)} g\rt + \lt T^{-(\cdot)}h\rt\bigg) 
          \Gamma(s,x,v), \label{G1_step1}\\
        & \bigg| T^{-s} L_2[f,g,h,l,m](s) \bigg| \le 2^{2d-1} \lt T^{-(\cdot)}f\rt\ \lt T^{-(\cdot)}g\rt\ \lt T^{-(\cdot)}h\rt\ \lt T^{-(\cdot)}l\rt \nonumber\\
        & \hspace{4.3cm}\times \lt T^{-(\cdot)}m\rt  \Gamma(s,x,v), \label{L2_step1}\\
        & \bigg| T^{-s} G_2[g,h,k,l,m](s) \bigg| \le 2^{2d-1} \lt T^{-(\cdot)}g\rt\ \lt T^{-(\cdot)}h\rt\ \lt T^{-(\cdot)}k\rt\ \lt T^{-(\cdot)}l\rt \nonumber\\
        & \hspace{4.3cm}\times \lt T^{-(\cdot)}m\rt  \Gamma(s,x,v). \label{G2_step1}
    \end{align}    
\end{lemma}
\begin{proof}
First consider $L_1$. We divide and multiply by the appropriate exponential weights and assemble the components of the triple norm, and we bound the exponentials dependent on $v_3,v_4,v_5$ by 1 to obtain 
    \begin{align}
        \bigg| &T^{-s} L_1(s)\bigg| = \bigg|  \int_{\R^{5d}}\delta(\Sigma)\delta(\Omega)T^{-s}f(x,v,s)T^{-s}g(x + s(v - v_1),v_1,s)
        \nonumber\\
        &\hspace{2cm} \times T^{-s}h(x + s(v - v_2),v_2,s)
        T^{-s}k(x +s(v - v_3),v_3,s) \nonumber \\
        & \hspace{2cm}\times \bigg(T^{-s}l(x + s(v - v_4),v_4,s) + T^{-s}m(x + s(v - v_5),v_5,s)\bigg)dv_1dv_2dv_3dv_4dv_5\nonumber\\
        & \le 2\lt T^{-(\cdot)}f\rt \lt T^{-(\cdot)}g\rt \lt T^{-(\cdot)}h\rt \lt T^{-(\cdot)}k\rt\bigg(\lt T^{-(\cdot)}l\rt + \lt T^{-(\cdot)}m\rt\bigg)\nonumber\\
        &\qquad \times\int_{\R^{5d}} \delta(\Sigma)\delta(\Omega)e^{-\alpha(|x|^2 + |x + s(v - v_1)|^2 + |x + s(v - v_2)|^2) - \beta(|v|^2 + |v_1|^2 + |v_2|^2)}dv_1dv_2dv_3dv_4dv_5.\label{L1}
    \end{align}
    Next we apply Lemma \ref{lem:I},
    and use that for $\boldsymbol{\omega}=(\omega_1,\omega_2) \in \mathbb{S}^{2d-1}$, $1 +\omega_1 \cdot \omega_2 \ge 1/2$ and $|\omega_1 \cdot (v_1 - v) + \omega_2\cdot (v_2 - v)| \le |v_1-v| +|v_2-v|$. So, we have 
    \begin{align}
        \bigg| &T^{-s} L_1(s) \bigg| 
        \le 2^{2d-1}\lt T^{-(\cdot)}f\rt\ \lt T^{-(\cdot)}g\rt\ \lt T^{-(\cdot)}h\rt\ \lt T^{-(\cdot)}k\rt \nonumber\\
        & \hspace{1cm}\times \bigg(\lt T^{-(\cdot)}l\rt + \lt T^{-(\cdot)}m\rt\bigg) e^{-\alpha|x|^2 - \beta|v|^2} \nonumber\\
        &\hspace{1cm} \times\int_{\R^{2d}}(|v_1 - v| + |v_2 - v|)^{2d-2} e^{-\alpha(|x + s(v - v_1)|^2 + |x + s(v - v_2)|^2) - \beta(|v_1|^2 + |v_2|^2)}dv_1dv_2,
    \end{align}   
which completes the proof of \eqref{L1_step1}. The estimate \eqref{L2_step1} for $L_2$ is done analogously.

Now we focus on proving the estimates \eqref{G1_step1} for the operator $G_1$.  First we notice the  following consequence of conservation of momentum and energy:
\begin{align}\label{identity}
    |x +& s(v - v_3)|^2 + |x + s(x - v_4)|^2 + |x + s(v - v_5)|^2 \nonumber \\
    &= 3|x|^2 + 2s x\cdot( 3v - v_3 - v_4 - v_5) 
    + s^2\left(3|v|^2 - 2v\cdot (v_3 + v_4 + v_5) + |v_3|^2 + |v_4|^2 + |v_5|^2\right)\nonumber\\
    & = 3|x|^2 + 2s x\cdot (2v - v_1 - v_2) 
    + s^2 \left(3|v^2| - 2v\cdot(v + v_1 + v_2) + |v|^2 + |v_1|^2 + |v_2|^2\right)\nonumber\\
    & = 3|x|^2 + 2s x\cdot ((v - v_1) + (v - v_2))
    + s^2(2|v|^2 - 2v\cdot(v_1 + v_2) + |v_1|^2 + |v_2|^2)\nonumber\\
   & = |x|^2 + |x + s(v - v_1)|^2 + |x  + s(v - v_2)|^2.
\end{align}

Just as in the proof of \eqref{L1_step1}, we multiply and divide by the appropriate exponentials to assemble the components of the triple norm. Then we have
\begin{align*}
    &\left\lvert T^{-s}G_1ds\right\rvert \\
    &= \bigg|\int_{\R^{5d}}\delta(\Sigma)\delta(\Omega)T^{-s}k(x + s(v - v_3), v_3,s)T^{-s}l(x + s(v - v_4), v_4,s)T^{-s}m(x + s(v - v_5),v_5,s)\\
    & \qquad \times \bigg[T^{-s}f(x,v,s)\bigg(T^{-s}g(x + s(v - v_1),v_1,s) + T^{-s}h(x + s(v - v_2),v_2,s)\bigg)\bigg]dv_1dv_2dv_3dv_4dv_5\bigg|\\
    & \leq \lt T^{-(\cdot)}k\rt\ \lt T^{-(\cdot)}l\rt\ \lt T^{-(\cdot)}m\rt\ \lt T^{-(\cdot)}f\rt\bigg(\lt T^{-(\cdot)}g\rt + \lt T^{-(\cdot)}h\rt\bigg)\\
    & \qquad \times \int_{\R^{5d}}\delta(\Sigma)\delta(\Omega)e^{-\alpha(|x + s(v - v_3)|^2 + |x + s(v - v_4)|^2 + |x + s(v - v_5)|^2) - \beta(|v_3|^2 + |v_4|^2 + |v_5|^2)}\\
    &\qquad \times e^{-\alpha|x|^2 - \beta|v|^2}\bigg(e^{-\alpha|x +s(v - v_1)|^2 - \beta|v_1|^2} + e^{-\alpha|x + s(v - v_2)|^2 - \beta|v_2|^2}\bigg)dv_1dv_2dv_3dv_4dv_5.
\end{align*}
We bound the last three exponentials  by 1, and apply \eqref{identity}  and conservation of energy encoded in $\Omega$ to obtain
\begin{align}
   &\left\lvert T^{-s}G_1ds\right\rvert \nonumber\\
   &\le 2\lt T^{-(\cdot)}k\rt\ \lt T^{-(\cdot)}l\rt\ \lt T^{-(\cdot)}m\rt\ \lt T^{-(\cdot)}f\rt\bigg(\lt T^{-(\cdot)}g\rt + \lt T^{-(\cdot)}h\rt\bigg)\nonumber\\
    &\qquad \times \int_{\R^{5d}}\delta(\Sigma)\delta(\Omega)e^{-\alpha(|x|^2 + |x +s(v - v_1)|^2 + |x + s(v - v_2)|^2) - \beta(|v|^2 + |v_1|^2 + |v_2|^2)}dv_1dv_2dv_3dv_4dv_5.\nonumber
\end{align}
Again, by applying Lemma \ref{lem:I},
    and using that for $\boldsymbol{\omega}=(\omega_1,\omega_2) \in \mathbb{S}^{2d-1}$, we have $1 +\omega_1 \cdot \omega_2 \ge 1/2$ and $|\omega_1 \cdot (v_1 - v) + \omega_2\cdot (v_2 - v)| \le |v_1-v| +|v_2-v|$, we get
    \begin{align}
        \bigg| &T^{-s} G_1(s) \bigg| \le 2^{2d-1} \lt T^{-(\cdot)}f\rt\ \lt T^{-(\cdot)}k\rt\ \lt T^{-(\cdot)}l\rt\ \lt T^{-(\cdot)}m\rt \nonumber\\
        & \hspace{1cm}\times \bigg(\lt T^{-(\cdot)}g\rt + \lt T^{-(\cdot)}h\rt\bigg) \, \Gamma(s,x,v), 
    \end{align}    
which completes the proof of \eqref{G1_step1}. The estimate \eqref{G2_step1} for $G_2$ is done analogously.
\end{proof}

\begin{proposition}\label{prop:DCT}
    Let $d\geq 1$. Let $\alpha,\beta > 0$ and suppose that functions $f,g,h,k,l,m \in \mathcal{F}$ are such that $T^{-(\cdot)}f, \,T^{-(\cdot)}g, \,T^{-(\cdot)}h, \, T^{-(\cdot)}k, \,T^{-(\cdot)}l, \,T^{-(\cdot)}m \in \M$. Then  we have
    \begin{align}
        & \left\|T^{-s}G_1[f,g,h,k,l,m](s)\right\|_{L^1_{x,v}} \in L^1(\R),\\
        &\left\|T^{-s}L_1[f,g,h,k,l,m](s)\right\|_{L^1_{x,v}} \in L^1(\R),\\
         & \left\|T^{-s}G_2[g,h,k,l,m](s)\right\|_{L^1_{x,v}} \in L^1(\R),\\
        &\left\|T^{-s}L_2[f,g,h,l,m](s)\right\|_{L^1_{x,v}} \in L^1(\R).
    \end{align}
\end{proposition}
\begin{proof} Due to Lemma \ref{lem:step1}, it suffices to prove that 
\begin{align}
     \|\Gamma(s,x,v)\|_{L^1_{x,v}} \in L^1(\R).
\end{align}
Note that if we use notation   $\mathbf{x} = \begin{pmatrix}
    x\\x
\end{pmatrix}$ and $\mathbf{u} = \begin{pmatrix}
    v_1 - v\\
    v_2 - v
\end{pmatrix}$, 
as was the case in the study of the binary-ternary Boltzmann equation \cite{amga22}, 
we can write
\begin{align*}
    \int_\R  \|\Gamma(s,x,v)\|_{L^1_{x,v}}  ds 
     =  & \int_{\R^{2d}}e^{-\alpha|x|^2 - \beta|v|^2}
    \int_{\R^{2d}}(|v_1 - v| + |v_2 - v|)^{2d-2}e^{-\beta(|v_1|^2 + |v_2|^2)} \\
    &\times\int_\R e^{-\alpha|\mathbf{x} + s\mathbf{u}|^2}ds dv_1dv_2dxdv.      
\end{align*}
By applying Lemma \ref{lem:time} and then Lemma \ref{lem:conv}, 
we have
\begin{align*}
\int_\R  \|\Gamma(s,x,v)\|_{L^1_{x,v}}  ds
    & \lesssim \alpha^{-\frac{1}{2}}  \int_{\R^{2d}}e^{-\alpha|x|^2 - \beta|v|^2}
        \int_{\R^{2d}}(|v_1 - v| + |v_2 - v|)^{2d-3}e^{-\beta(|v_1|^2 + |v_2|^2)}dv_1dv_2dxdv \\
    & \lesssim C_{\beta,d} \, \alpha^{-1/2} \int_{\R^{2d}}e^{-\alpha|x|^2 - \beta|v|^2} (1 + |v|^{(2d-3)^+})dxdv \\
    &\le C_{\alpha, \beta, d}.
\end{align*}
where $C_{\beta,d}$ is defined in \eqref{conv_constant} and $C_{\alpha, \beta, d}$ is a positive constant depending on $\alpha, \beta$ and $d$.
This concludes the proof of the proposition.
\end{proof}

Our proofs of Theorems \ref{thm:wke well posed}, \ref{thm:nonnegativity} and \ref{thm: scattering 1} rely on the following estimate in the triple norm of the time averages of the transported $G$ and $L$ operators.

 \begin{proposition}\label{prop:gl_bounds}
  Let $d = 1$. Let $\alpha, \beta >0$ and suppose that  functions $f,g,h, k,l, m \in \mathcal{F}$ are such that $T^{-(\cdot)}f, \,T^{-(\cdot)}g,\, T^{-(\cdot)}h, \,T^{-(\cdot)}k,\, T^{-(\cdot)}l, \,T^{-(\cdot)}m \in \M$. Then for any $t,a \in {\R} \cup \{-\infty, \infty\}$, the following estimates hold
     \begin{align}
         &\ld \int_{t}^{a} T^{-s}G_1[f,g,h,k,l,m](s)ds\rd 
         \le 4C_{1,\beta} \alpha^{-1/2}\lt T^{-(\cdot)}f\rt\,\lt T^{-(\cdot)}k\rt\,\lt T^{-(\cdot)}l\rt\, \nonumber\\
         & \hspace{5cm} \times \lt T^{-(\cdot)}m\rt
        \bigg(\lt T^{-(\cdot)}g\rt + \lt T^{-(\cdot)}h\rt\bigg), \label{G1 estimate}\\
        & \ld \int_{t}^{a} T^{-s}G_2[g,h,k,l,m](s)ds\rd 
        \le 4 C_{1,\beta} \alpha^{-1/2}\lt T^{-(\cdot)}k\rt\,\lt T^{-(\cdot)}l\rt\,\lt T^{-(\cdot)}m\rt \nonumber\\
        & \hspace{5cm} \times \lt T^{-(\cdot)}g\rt\,\lt T^{-(\cdot)}h\rt, \label{G2 estimate}\\
        &\ld \int_{t}^{a} T^{-s}L_1[f,g,h,k,l,m](s)ds\rd  
        \le 4 C_{1,\beta} \alpha^{-1/2} \lt T^{-(\cdot)}f\rt\,\lt T^{-(\cdot)}g\rt\,\lt T^{-(\cdot)}h\rt \nonumber\\
        & \hspace{5cm} \times \lt T^{-(\cdot)}k\rt \bigg(\lt T^{-(\cdot)}l\rt + \lt T^{-(\cdot)}m\rt\bigg), \label{L1 estimate}\\
        &\ld \int_{t}^{a} T^{-s}L_2[f,g,h,l,m](s)ds\rd 
        \le 4 C_{1,\beta} \alpha^{-1/2}\lt T^{-(\cdot)}f\rt\,\lt T^{-(\cdot)}g\rt\,\lt T^{-(\cdot)}h\rt \nonumber\\
        & \hspace{5cm} \times 
        \lt T^{-(\cdot)}l\rt\,
        \lt T^{-(\cdot)}m\rt,\label{L2 estimate}
    \end{align}
$C_{1,\beta}$ is given \eqref{conv_constant}.
 \end{proposition}
 \begin{remark}
     We emphasize that the right-hand sides of the estimates in the above proposition do not depend on time. Consequently, we obtain same bounds when supremum in time is applied. For example, 
     \begin{align*}
         \sup_{t \in \R}\ld \int_{t}^{a} T^{-s}G_1[f,g,h,k,l,m](s)ds\rd 
         &\le 4C_{1,\beta} \alpha^{-1/2}\lt T^{-(\cdot)}f\rt\,\lt T^{-(\cdot)}k\rt\,\lt T^{-(\cdot)}l\rt\\
          &\qquad \times \lt T^{-(\cdot)}m\rt
        \bigg(\lt T^{-(\cdot)}g\rt + \lt T^{-(\cdot)}h\rt\bigg).
     \end{align*}
 \end{remark}
\begin{proof}
   By Lemma \ref{lem:step1}, it suffices to prove that
   \begin{align}\label{int_I}
       \ld \int_{t}^{a} \Gamma(s,x,v) ds \rd \le 2 C_{1,\beta} \alpha^{-1/2}.
   \end{align}

 We have
 \begin{align*}
     \int_{t}^{a} \Gamma(s,x,v) ds
      \le & e^{-\alpha|x|^2 - \beta|v|^2}
        \int_\R\int_{\R^{2d}}(|v_1 - v| + |v_2 - v|)^{2d-2} \\
       & \hspace{2cm}\times e^{-\alpha(|x + s(v - v_1)|^2 + |x + s(v - v_2)|^2) - \beta(|v_1|^2 + |v_2|^2)}dv_1dv_2ds.
 \end{align*}
     If we let $\mathbf{x} = \begin{pmatrix}
        x\\x
    \end{pmatrix}$ and $\mathbf{u} = \begin{pmatrix}
        v_1 - v\\
        v_2 - v
    \end{pmatrix}$,
    and apply Lemma \ref{lem:time} and then Lemma \ref{lem:conv}, we get
    \begin{align*}
     \int_{t}^{a} \Gamma(s,x,v) ds
       & \le 2 \alpha^{-1/2} e^{-\alpha|x|^2 - \beta|v|^2} \int_{\R^{2d}}(|v_1 - v| + |v_2 - v|)^{2d-3}e^{-\beta(|v_1|^2 + |v_2|^2)}dv_1dv_2\\
      & \leq  2C_{d,\beta} \alpha^{-1/2}(1 + |v|^{(2d-3)^+}) e^{-\alpha|x|^2 - \beta|v|^2}, 
      \end{align*} 
where $C_{d,\beta}$ is given \eqref{conv_constant}. Since $d = 1$, then $(2d-3)^+ = 0$, and so
\begin{align*}
       \int_{t}^{a} \Gamma(s,x,v) ds& \leq 2 C_{1,\beta} \alpha^{-1/2}e^{-\alpha|x|^2 - \beta|v|^2},
      \end{align*} where $C_{1,\beta}$ is given \eqref{conv_constant}. This implies \eqref{int_I}, and so the the proof is complete.
\end{proof}
\begin{remark}
Note that when we apply Lemma \ref{lem:conv}, we need that $2d - 3 \leq 0$ so that $\left\|\left|\int_{t}^{a} \Gamma(s,x,v)\right\|\right|_{\alpha,\beta}$ is bounded by a constant. This requires $d \leq \frac{3}{2}$, constraining the dimensions that we can consider to $d = 1$. 
\end{remark}

\begin{remark}
    As a consequence of the proof of Proposition \ref{prop:gl_bounds}, under the same assumptions as in the proposition, we have the following bound     
    \begin{align}\label{int_convergent}
        \int_\R \ld T^{-s}\C[f,g,h,k,l,m](s)\rd ds < \infty,
    \end{align}
which is used in Section \ref{sec:scattering}.
\end{remark}

Our proof of Theorem \ref{thm:KS_WP} relies on the following $L^1_{x,v}$ estimates of $L$ and $G$ operators as well as the immediate corollary.
\begin{proposition}\label{prop:l1_estimate}
     Let $d \geq 1$ and 
    $\alpha,\beta > 0$. Then there is a positive constant $C_{\alpha,\beta,d} = C(\alpha, \beta, d)$ such that the following hold: 
    For any $f,g,h,k,l,m \in \mathcal{F}^+$, with $T^{-(\cdot)} f, T^{-(\cdot)}  g, T^{-(\cdot)}  h, T^{-(\cdot)}  k, T^{-(\cdot)}  l$, $T^{-(\cdot)}  m \in L^\infty(\R,\m)$, and $t \in \R$, we have 
        \begin{align}
            \|T^{-t}L_1[f,g,h,k,l,m](t)\|_{L^1_{x,v}} &\leq C_{\alpha,\beta,d} \lt \T f \rt \lt \T g \rt \lt \T h \rt \lt \T k \rt \nonumber\\
            &\qquad \times \left(\lt \T l \rt + \lt \T m \rt \right),\label{l1_L1}\\
            \|T^{-t}L_2[f,g,h,l,m](t)\|_{L^1_{x,v}} &\leq C_{\alpha,\beta,d}\lt \T f \rt \lt \T g \rt \lt \T h \rt\nonumber\\
            &\qquad \times \lt \T l \rt \lt \T m \rt, \label{l1_L2}\\
            \|T^{-t}G_1[f,g,h,k,l,m](t)\|_{L^1_{x,v}} &\leq C_{\alpha,\beta,d} \lt \T k \rt \lt \T l \rt \lt \T m \rt \lt \T f \rt \nonumber\\
            &\qquad \times \left(\lt \T g \rt + \lt \T h \rt\right),\label{l1_G1}\\
            \|T^{-t}G_2[g,h,k,l,m](t)\|_{L^1_{x,v}} & \leq C_{\alpha,\beta,d} \lt \T g \rt \lt \T h \rt \lt \T k \rt \nonumber\\
            &\qquad \times \lt \T l \rt \lt \T m \rt.\label{l1_G2}
        \end{align}
        Moreover, \[\T G[f,g,h,k,l,m], ~\T L[f,g,h,k,l,m] \in L^\infty(\R,L^{1+}_{x,v}).\]
\end{proposition}
\begin{proof}
By Lemma \ref{lem:step1}, it suffices to prove that for any $t\in \R$ we have
\begin{align}\label{int_L1}
   \left\| \Gamma(t,x,v) \right\|_{L^1_{x,v}} \le C_{\alpha, \beta, d}. 
\end{align}
Bounding the exponentials dependent on $t$ in the definition of $\Gamma$ in \eqref{def:I} by 1, it is clear that 
\begin{align*}
     \left\| \Gamma(t,x,v) \right\|_{L^1_{x,v}}
      \le \int_{\R^{2d}}e^{-\alpha|x|^2 - \beta|v|^2}\int_{\R^{2d}}(|v - v_1| + |v - v_2|)^{2d-2}e^{-\beta(|v_1|^2 + |v_2|^2)}dv_1dv_2dxdv.
\end{align*}
Then, calculating the integral in $x$ and applying Lemma \ref{lem:conv} yields
        \begin{align}
            \left\| \Gamma(t,x,v) \right\|_{L^1_{x,v}}
             &\leq C_\alpha  \int_{\R^{2d}}(1 + |v|^{(2d-2)^+})e^{-\beta|v|^2}dv
             \le C_{\alpha,\beta,d}, 
    \end{align}
     since $(1 + |v|^{(2d -2)^+})e^{-\beta|v|^2} \in L^1_v$. This completes the proof.
\end{proof}

Proposition \ref{prop:l1_estimate} also implies an $L^1_{x,v}$-continuity for the transported $L$ and $G$ operators, which is stated in the following corollary.
\begin{corollary}\label{cor:l1_conv}
    Let $d\ge 1$ and
    $\alpha,\beta > 0$. For $i \in \{1,2,3,4,5,6\}$, consider  sequences $\{f^{i,n}\}_n \subset \mathcal{F}^+$ with $T^{-t}f^{i,n}(t) \in \m$ for all $t\in \R$,  and functions $f^i \in \mathcal{F}^+$ such that $T^{-t} f^{i,n}(t) \overset{\m}{\longrightarrow} T^{-t}f^i(t)$, for all $t \in \R$. Then, for all $t \in \R$, the following convergence holds, 
    \begin{align*}
        \bigg(T^{-t}L[f^{1,n},f^{2,n},f^{3,n},f^{4,n},f^{5,n},f^{6,n}](t), \,\, T^{-t}G[f^{1,n},f^{2,n},f^{3,n},f^{4,n},f^{5,n},f^{6,n}](t)\bigg) \\
    \overset{L^1_{x,v}}{\longrightarrow} \bigg(T^{-t}L[f^1,f^2,f^3,f^4,f^5,f^6](t), \,\, T^{-t}G[f^1,f^2,f^3,f^4,f^5,f^6](t)\bigg).
    \end{align*}
\end{corollary}
\begin{proof} It is sufficient to show $L^1_{x,v}$ continuity for each component of the $L$ and $G$ terms. We consider the case of $L_1$ as the arguments for $L_2, G_1,$ and $G_2$ will follow identically. 
Using  the decomposition \eqref{L1 decomp} of $L_1$, the triangle inequality and the estimate \eqref{l1_L1}, we have
 \begin{align*}
    &\left\|T^{-t}L_1[f^{1,n},f^{2,n},f^{3,n},f^{4,n},f^{5,n},f^{6,n}](t) - T^{-t}L_1[f^1,f^2,f^3,f^4,f^5,f^6](t)\right\|_{L^1_{x,v}}\\
        & \lesssim  C\bigg[
        \left(\ld T^{-t}f^{5,n}\rd + \ld T^{-t}f^{6,n}\rd\right)\\
        & \hspace{1cm} \times\Big(
        \ld T^{-t}f^{1,n} - T^{-t}f^1\rd \ld T^{-t}f^{2,n}\rd \ld T^{-t}f^{3,n}\rd \ld T^{-t}f^{4,n}\rd \\
        &\hspace{1.5cm} + \ld T^{-t}f^{1}\rd \ld T^{-t}f^{2,n} - T^{-t}f^2\rd \ld T^{-t}f^{3,n}\rd \ld T^{-t}f^{4,n}\rd\\
        &\hspace{1.5cm} + \ld T^{-t}f^{1}\rd \ld T^{-t}f^{2}\rd \ld T^{-t}f^{3,n} - T^{-t}f^3\rd \ld T^{-t}f^{4,n}\rd\\
        &\hspace{1.5cm} + \ld T^{-t}f^{1}\rd \ld T^{-t}f^{2}\rd \ld T^{-t}f^{3}\rd \ld T^{-t}f^{4,n} - T^{-t}f^4\rd\Big)\\
        &\hspace{0.8cm} + \ld T^{-t}f^{1}\rd \ld T^{-t}f^{2}\rd \ld T^{-t}f^{3}\rd \ld T^{-t}f^{4}\rd\\
        &\hspace{2cm} \times \left(\ld T^{-t}f^{5,n} - T^{-t}f^5\rd + \ld T^{-t}f^{6,n} - T^{-t}f^6\rd\right)\bigg]\\
        & \lesssim C\bigg(\ld T^{-t}f^{1,n} - T^{-t}f^1\rd + \ld T^{-t}f^{2,n} - T^{-t}f^2\rd + \ld T^{-t}f^{3,n} - T^{-t}f^3\rd \\
        &\qquad  + \ld T^{-t}f^{4,n} - T^{-t}f^4\rd + \ld T^{-t}f^{5,n} - T^{-t}f^5\rd + \ld T^{-t}f^{6,n} - T^{-t}f^6\rd \bigg),
    \end{align*}
   
where $C$ is a positive constant depending on $\alpha, \beta, d$ and the maximum of the norms of all involved functions (which is finite thanks to the assumptions of the corollary).
    Since $T^{-t}f^{i,n} \overset{\m}{\longrightarrow} T^{-t}f^i$ for all $t \in \R$, each of these terms goes to zero as $n\to \infty$ and we obtain the desired convergence. 
\end{proof}
\subsection{A key mapping}\label{subsec:prep} 
We construct several mappings in the proofs of global well-posedness and existence of non-negative solutions (Section \ref{sec:GWP}) and scattering (Section \ref{sec:scattering}) that we wish to show are contractions. To streamline these arguments, we define the following  mapping $\Lambda_{t,a}$ with parameters $a,b \in \overline{\R} = \R \cup\{-\infty, \infty\}$:
 \begin{align}
    \Lambda_{a,b}[g] = \int_{a}^{b}T^{-s}\C[T^sg](s)~ds. \label{Lambda}
\end{align}
\begin{lemma}\label{lem:lambda}
    Let $d=1$, $\alpha,\beta > 0$, $t,a \in \overline{\R}$ and $\rho \leq \frac{\alpha^{1/8}}{2^{\frac{5}{2}}C_{1,\beta}^{1/4}}$, where $C_{1,\beta}$ is defined in \eqref{conv_constant}.  Then for $g,h \in \overline{B}(\rho) = \{f \in \mt: \lt f\rt \le \rho\}$, we have  
    \begin{align}
        \lt \Lambda_{t,a}[g]\rt &\leq \frac{\rho}{32},\label{lambda_bound}\\
        \lt \Lambda_{t,a}[g] - \Lambda_{t,a}[f]\rt &\leq \frac{1}{4}\lt g - h\rt, \label{lambda_diff_bound}
    \end{align} 
    where $ \lt \Lambda_{t,a}[g]\rt := \sup_{t\in\R}  \ld  \Lambda_{t,a}[g] \rd$.
\end{lemma}
\begin{proof}
First, we show \eqref{lambda_bound}.
Let $g \in \overline{B}(\rho)$. By Proposition \ref{prop:gl_bounds}, 
\begin{align*}
    \lt \Lambda_{t,a}[g]\rt 
    & \leq \lt \int_{t}^{a} T^{-s}G_1[T^sg](s)~ds\rt + \lt \int_{t}^{a} T^{-s}G_2[T^sg](s)~ds\rt \\
    &\qquad+ \lt \int_{t}^{a} T^{-s}L_1[T^sg](s)~ds\rt + \lt \int_{t}^{a} T^{-s}L_2[T^sg](s)~ds\rt\\
    &\lesssim 6\cdot 4C_{1,\beta}\alpha^{-1/2}\lt g\rt^5
     \leq  2^{5}C_{1,\beta} \alpha^{-1/2}\rho^5
     \leq \frac{\rho}{32},
\end{align*}
where the last inequality is a consequence of the assumed upper bound on $\rho$.

Next we show the inequality \eqref{lambda_diff_bound}. Let $g,h \in \overline{B}(\rho)$. Then applying  decompositions \eqref{G1 decomp}- \eqref{L2 decomp} and Proposition \ref{prop:gl_bounds}, we have 
\begin{align*}
    & \lt \Lambda_{t,a}[g] - \Lambda_{t,a}[h]\rt\\ 
    & \leq \lt \int_{t}^{a} T^{-s}G_1[T^sg](s) - T^{-s}G_1[T^sh](s)~ds\rt 
    + \lt \int_{t}^{a} T^{-s}G_2[T^sg](s) - T^{-s}G_2[T^sh](s)~ds\rt\\
    & \quad + \lt \int_{t}^{a} T^{-s}L_1[T^sg](s) - T^{-s}\L_1[T^sh](s)~ds\rt 
    + \lt \int_{t}^{a} T^{-s}L_2[T^sg](s) - T^{-s}L_2[T^sh](s)~ds\rt\\
    &\leq 6\cdot 4 C_{1,\beta} \alpha^{-1/2}\lt g - h \rt  \bigg[\lt g\rt^4 + \lt g \rt^3 \cdot \lt h \rt + \lt g \rt^2 \cdot \lt h \rt^2 \\
    &\hspace{4.2cm} + \lt g \rt\cdot \lt h \rt^3 + \lt h \rt^4\bigg]\\
    &\leq 2^{7} C_{1,\beta}\alpha^{-1/2}\rho^4\lt g - h\rt\\
    &\leq \frac{1}{4}\lt g - h \rt.
\end{align*}
\end{proof}

\section{Global well-posedness}\label{sec:GWP}
In this section we prove well-posedness of the 6-wave kinetic equation \eqref{eq:6wave} as stated in Theorem \ref{thm:wke well posed} via the fixed point argument.

\begin{proof}[Proof of Theorem \ref{thm:wke well posed}]
Motivated by our definition of mild solution (Definition \ref{def:solution}), we define the mapping $\mathcal{T}:\mt \rightarrow \mt$ by 
\begin{align}
\mathcal{T}[g(t)] = f_0 + \Lambda_{0,t}[g],
\end{align}
where $\Lambda_{0,t}$ is given by \eqref{Lambda}. The mapping $\mathcal{T}:\mt \rightarrow \mt$ is well-defined by Proposition \ref{prop:gl_bounds}. 

Let $R_e$ be as in the statement of the theorem, and recall 
\begin{align}
    \overline{B}(2R_e) = \left\{g \in \mt: \lt g\rt \leq 2R_e\right\}.
\end{align} 
We claim that $\mathcal{T}: \overline{B}(2R_e)\rightarrow \overline{B}(2R_e)$ is a contraction. Namely, for $g \in \overline{B}(2R_e)$, Lemma \ref{lem:lambda} with $\rho = 2R_e$ implies
    \begin{align}
        \lt \mathcal{T}[g]\rt &\leq \ld f_0\rd + \lt \Lambda_{0,t}[g]\rt 
        \leq R_e + \frac{1}{16}R_e 
        \leq 2R_e.
    \end{align}  
Thus, $\mathcal{T}:\overline{B}(2R_e) \to \overline{B}(2R_e)$. 
Next, we show that $\mathcal{T}$ is a contraction on $\overline{B}(2R_e)$. For $g, h \in \overline{B}(2R_e)$, by Lemma \ref{lem:lambda} we have 
    \begin{align} 
        \lt \mathcal{T}[g] - \mathcal{T}[h]\rt 
        & = \lt \Lambda_{0,t}[g] - \Lambda_{0,t}[h]\rt
         \leq \frac{1}{4}\lt g - h \rt.\label{new}
    \end{align}
Thus, $\mathcal{T}:\overline{B}(2R_e) \rightarrow \overline{B}(2R_e)$ is a contraction. 
Then $\mathcal{T}$ has a unique fixed point, $g \in \overline{B}(2R_e)$, so
\begin{align}\label{eq:fixed-g}
    g(t) = \mathcal{T}[g(t)].
\end{align}
By setting
\begin{align}
f(t) = T^{t}g(t)
\end{align}
equation \eqref{eq:fixed-g} becomes
\begin{align}
    T^{-t} f(t) =  f_0 + \int_0^t T^{-s}\C[f](s)ds.
\end{align}
 Thanks to the fact that $g \in \overline{B}(2R_e)$, $f$ satisfies the bound 
\begin{align}
  \lt T^{-t}f(t)\rt \leq 2R_e,
\end{align}
which by Proposition \ref{prop:gl_bounds} shows that $f(t)$ is the unique  mild solution of (\ref{eq:6wave}) in $\overline{B}(2R_e)$.

Finally, we prove the stability estimate \eqref{stability}. Let $f_0,g_0$ be such that  $||f_0||_{\alpha, \beta}, ||g_0||_{\alpha, \beta} < R_e$, and let $f,g$ be mild solutions of \eqref{eq:6wave} corresponding to initial data $f_0,g_0$, respectively. Then we have 
\begin{align}
    T^{-t}f(t) - T^{-t}g(t) = f_0 - g_0 &+ \int_0^t \Big((T^{-s}\C[f](s) - T^{-s}\C[g](s)\Big)ds.\nonumber
\end{align}
Applying the estimate \eqref{lambda_diff_bound} yields
\begin{align}
     \lt T^{-t}f(t) - T^{-t}g(t)\rt &\leq \ld f_0 - g_0\rd + \frac{1}{4}\lt T^{-t}f(t) - T^{-t}g(t)\rt,
\end{align}
which in turn implies
\begin{align}
   \lt T^{-t}f(t) - T^{-t}g(t)\rt & \leq 2\ld f_0 - g_0\rd.
\end{align}
This concludes the proof of Theorem \ref{thm:wke well posed}. 
\end{proof}

Next, we apply additional constraints on $\alpha$ and $\beta$ to show the existence of nonnegative solutions. The argument is inspired by analysis on the Boltzmann equation done in \cite{baga16}.

\begin{proof}[Proof of Theorem \ref{thm:nonnegativity}]
Consider the same mapping as in the proof of Theorem \ref{thm:wke well posed}. Namely, let $\mathcal{T}:\mt \rightarrow \mt$ be given by 
\begin{align}
\mathcal{T}[g(t)] = f_0 + \int_0^t T^{-s}\C[T^s g](s)ds.
\end{align}

We will show that the mapping $\mathcal{T}$ is a contraction in a ball centered at the Maxwellian, $M$, i.e. in the following ball
\begin{align}
\overline{B}_M(2R_p): = \{g \in \mt: \lt g - M \rt \le 2R_p \}.
\end{align}

First we show that $\mathcal{T}: \overline{B}_M(2R_p)\rightarrow \overline{B}_M(2R_p)$. Let $g \in \overline{B}_M(2R_p)$. Then, we have
\begin{align*}
    \lt \mathcal{T}[g] - M \rt 
    & = \lt f_0 - M + \int_0^t T^{-s}\C[T^s g](s)ds \rt 
     \le R_p + \lt\int_0^t T^{-s}\C[T^s g](s)ds \rt.
\end{align*}

From the estimates given in Proposition \ref{prop:gl_bounds},  
\begin{align*}
    \lt \mathcal{T}[g] - M \rt  &\leq R_p + 6\cdot 4 C_{1,\beta}\alpha^{-1/2}\lt g\rt^5\\
    &\leq R_p + 24 C_{1,\beta} \alpha^{-1/2}\left(\lt g - M \rt + \lt M \rt\right)^5\\
    & \le R_p + 24 C_{1,\beta} \alpha^{-1/2} (2R_p +1)^5\\
    & \le 2R_p
\end{align*}
as long as
\begin{align}
24 C_{1,\beta} \alpha^{-1/2} (2R_p +1)^5 \le R_p. \label{first}
\end{align}

Next, we show that $\mathcal{T}$ is a contraction on $\overline{B}_M(2R_p)$. Let $g,h \in \overline{B}_M(2R_p)$. Using \eqref{G1 decomp} - \eqref{L2 decomp}, Proposition \ref{prop:gl_bounds} and adding and subtracting $M$, we have
\begin{align*}
    & \lt \mathcal{T}[g] - \mathcal{T}[h]\rt = \lt \int_0^t T^{-s}\C[T^sg](s) - T^{-s}\C[T^s h](s)ds \rt\\
    & \leq 6\cdot 4C_{1,\beta} \alpha^{-1/2}\lt g - h\rt \bigg[(\lt g - M \rt + 1)^4 + (\lt g - M \rt + 1)^3(\lt h - M\rt  + 1) \\
    &\qquad + (\lt g - M \rt + 1)^2(\lt h - M \rt  + 1)^2 + (\lt g - M \rt + 1)(\lt h - M \rt + 1)^3 \\
    &\qquad + (\lt h - M \rt + 1)^4\bigg]\\
     &\leq 2^{7} C_{1,\beta} \alpha^{-1/2}(2R_p + 1)^4\lt g - h \rt\\
     & \leq \frac{1}{2}\lt g - h\rt,
\end{align*}
where the last inequality holds provided that 
\begin{align}
2^{7} C_{1,\beta} \alpha^{-1/2} (2R_p +1)^4 < \frac{1}{2}.\label{second}
\end{align}

From the conditions of the theorem, we have that \begin{align}\label{condition}
    \frac{1}{6} \leq R_p \leq \frac{1}{2} \left( \frac{1}{\sqrt[4]{2^{8} C_{1,\beta} \alpha^{-1/2}}} - 1\right).
\end{align} The upper bound on $R_p$ tells us that \eqref{second} is satisfied. 
We note that the condition \eqref{condition 1} is necessary so that the upper and lower bounds on $R_p$ define a non-empty interval.

Next we prove \eqref{first}. From \eqref{second}, we have
\begin{align*}
    (2R_p + 1)^4 < \frac{\alpha^{1/2}}{2^{8}C_{1,\beta}}.
\end{align*}
Therefore, the left-hand side of \eqref{first} is bounded as follows:
\begin{align*}
    24 C_{1,\beta} \alpha^{-1/2} (2R_p +1)^5 
    & \le 2^{5} C_{1,\beta} \alpha^{-1/2}  \frac{\alpha^{1/2}}{2^{8}C_{1,\beta}} (2R_p +1)\\
    & \le \frac{1}{8} (2R_p +1) \le R_p
\end{align*}
since by the assumption \eqref{condition} we have $R_p \ge 1/6$.

Then, we have that $\mathcal{T}:\overline{B}_M(2R_p) \to \overline{B}_M(2R_p)$ is a contraction. So, by the contraction mapping principle, $\mathcal{T}$ has a unique fixed point, $g \in \overline{B}_M(2R_p)$:
\begin{align}\label{eq:fixed-gM}
    g(t) = \mathcal{T}[g(t)].
\end{align}
By setting 
\begin{align}
    f(t) = T^{t}g(t)
\end{align}
equation \eqref{eq:fixed-gM} becomes 
\begin{align}
    T^{-t}f(t) = f_0 + \int_0^t T^{-s}\C[f]ds.
\end{align}
Thanks to the fact that $g \in \overline{B}_M(2R_p)$, $f$ satisfies the bound 
\begin{align} \label{set}
  \lt T^{-t}f(t) - M \rt < 2R_p,   
\end{align} 
which implies $\lt T^{-t}f(t)\rt \le 2R_p +1$, so $f\in L^\infty(\R, \m)$. Together with 
Proposition \ref{prop:gl_bounds} we have that $f(t)$ is the unique  mild solution of (\ref{eq:6wave}) in $\overline{B}_M(2R_p)$.
 This proves statement (a).

If in addition $\frac{1}{2^{12}} \leq C_{1,\beta} \alpha^{-1/2}$, then thanks to \eqref{condition 2} we have $R_p \leq \frac{1}{2} \left( \frac{1}{2^2\sqrt[4]{ C_{1,\beta} \alpha^{-1/2}}} - 1\right) \leq \frac{1}{2}$.  
Then, due to \eqref{set}, for all $t \in \R$, we have
\begin{align*}
 0 \leq (1-2R_p)M(x,v) \leq T^{-t}f(t,x,v) \leq (1 + 2R_p)M(x,v),
\end{align*}
which concludes the proof of statement (b). 
\end{proof}

\section{Kaniel-Shinbrot iteration scheme}
\label{sec:KS}

In this section we prove Theorem \ref{thm:KSnonnegative}. More precisely, we adapt the Kaniel-Shinbrot iteration to prove the existence and uniqueness of a non-negative mild solution of the 6-wave kinetic equation \eqref{eq:6wave} without additional constraints on $\alpha$ and $\beta$.

\begin{remark}
    The Kaniel-Shinbrot iteration used in prior papers (e.g. \cite{ alga09, amga22}) typically denotes transport operator by \#, i.e.:
    \begin{align*}
        T^{-t}f(t,x,v) = f^\#(t,x,v) := f(t,x +tv,v).
    \end{align*}
\end{remark}

We recall the abbreviated notation \eqref{abbr} that will be used throughout this section:
\begin{align}
    T^{-(\cdot)}g := T^{-t} g(t,x,v).
\end{align}

\subsection{Iteration set up} Given initial data $f_0(x,v) \in \m^+$ for  \eqref{eq:6wave} and given a pair of functions  $T^{-t}l_0(t,x,v) , T^{-t}u_0(t,x,v) \in L^\infty(\R,\m^+)$ that will be chosen to satisfy certain condition (see Theorem \ref{thm:KS_WP}), we introduce a generalization of the  Kaniel-Shinbrot iteration, which for $n\in\N$ reads
\begin{align}
    \frac{dl_n}{dt} + v\cdot \nabla_x l_n  &= G[l_{n-1}] -   L[l_n, u_{n-1}, u_{n-1}, u_{n-1}, u_{n-1}, u_{n-1}],\nonumber\\
    l_n(0) &= f_0\label{ln},\\
    \frac{du_n}{dt} + v\cdot\nabla_x u_n &= G[u_{n-1}] - L[u_n, l_{n-1}, l_{n-1}, l_{n-1}, l_{n-1}, l_{n-1}],\nonumber\\
    u_n(0) & = f_0, \label{un}
\end{align}
where $G$ is defined in \eqref{def-g}, $L$ is defined in \eqref{def-l}, and we use abbreviated notation:
\begin{align}
    G[p] = G[p,p,p,p,p,p].
\end{align}

We will show that the above iterative scheme yields an increasing in $n$ sequence $\{l_n(t)\}_n$ and decreasing in $n $ sequence $\{u_n(t)\}_n$ that satisfy 
\begin{align}
    T^{-(\cdot)}l_0 \le T^{-(\cdot)}l_1 \le \dots \le  T^{-(\cdot)}l_n \le T^{-(\cdot)} u_n \le \dots \le T^{-(\cdot)}u_1 \le T^{-(\cdot)}u_0.
\end{align} 
We will show that these two sequences, $\{l_n\}_n$ and $\{u_n\}_n$, converge to the same limit $f$, which will be shown to solve the initial value problem for the 6-wave kinetic equation \eqref{eq:6wave}.  Finally, $f$ will be  non-negative because $l_0$ will be chosen to be zero (see Subsection \ref{sec:application of KS}).

\subsection{Associated Linear Problem}
In order to establish that the above iteration has a solution, we  consider the following more general associated linear problem. Roughly speaking, given functions $g$ and $h$, we want to show well-posedness of the linear problem 
\begin{align}\label{ALP}
    \begin{cases}
        \partial_t f  + v\cdot \nabla_x f= h - L[f,g,g,g,g,g], & (t,x,v) \in \R\times \R^d \times \R^d\\
        f(0) = f_0, & (x,v) \in \R^d \times \R^d,
    \end{cases}
\end{align}
where $L$ is defined  in \eqref{def-l}. More precisely, we define a mild solution of \eqref{ALP} as follows.

\begin{definition}\label{alp_soln}
    Let $d\ge 1, \alpha,\beta > 0$ and  let $L$ be as in \eqref{def-l}. Suppose $f_0 \in L^{1,+}_{x,v}$, $\T h \in L_{loc}^1(\R,L^{1,+}_{x,v})$, $\T g \in L^\infty(\R, \m^+)$. Then a non-negative measurable function $f$ with $T^{-(\cdot)}f \in C^0(\R, L^{1+}_{x,v})$ that satisfies
     \begin{align*}
            T^{-t}f(t,x,v)  = f_0(x,v) + \int_0^t T^{-\tau}h(\tau,x,v)d\tau -  \int_0^t T^{-\tau}L[f,g,g,g,g,g](\tau,x,v)d\tau,
        \end{align*}
is a mild solution to $\eqref{ALP}$.
\end{definition}

By minor modifications  of  the arguments in \cite[Proposition 4.5]{amga22} to accommodate the higher order of multi-linearity of 
$ L(f,g,g,g,g,g)$,
we have the following existence and uniqueness result for the associated linear problem \eqref{ALP}.
\begin{proposition}\label{prop:ALP}
    Let $d\ge 1, \alpha,\beta > 0$. Consider  initial data $f_0 \in L^{1,+}_{x,v}$, $\T h \in L^1_{loc}(\R,L^{1+}_{x,v})$ and $ \T g \in L^\infty(\R, \m^+).$  Then there exists a unique mild solution $f$ of \eqref{ALP}.
\end{proposition}
Also, as in \cite[Corollary 4.6]{amga22} , we have the following comparison result.
\begin{corollary}[] \label{cor:ineq}
    Let 
    $d\ge 1, \alpha,\beta > 0$. Consider $f_{0,1},f_{0,2} \in L^{1,+}_{x,v}$, $\T h_1, \T h_2 \in L^1_{loc}(\R,L^{1,+}_{x,v})$ and $\T g_1, \T g_2 \in L^\infty(\R,\m^+)$  with 
    \[
    f_{0,1}\leq f_{0,2},\quad \T g_1 \geq \T g_2, \quad \T h_1 \leq \T h_2.
    \]
    Let $f_1,f_2$ be the corresponding unique solutions of \eqref{ALP} with $f_0 = f_{0,i}, g = g_i$ and $h = h_i$ for $i = 1,2$. Then $f_1 \leq f_2$. 
\end{corollary}
\subsection{Kaniel-Shinbrot Theorem} 
We now prove convergence of the Kaniel-Shinbrot iteration \eqref{ln}-\eqref{un} to the unique solution of the Cauchy problem \eqref{eq:6wave}. 

Let  
$\alpha,\beta > 0$. Consider 
\begin{itemize}
    \item initial datum $f_0(x,v) \in \m^+$ and
    \item a pair of functions  $T^{-(\cdot)}l_0, T^{-(\cdot)}u_0 \in L^\infty(\R,\m^+)$.
\end{itemize} Then by the a priori estimates in Proposition \ref{prop:l1_estimate}, we conclude that 
\begin{align} 
\T G[l_0],\,  \T G[u_0] \in L^1_{loc}(\R, L^{1+}_{x,v}).
\end{align}
We can now apply Proposition \ref{prop:ALP}  with $g$ as either $l_0$ or $u_0$ and $h$ as either $G[l_0]$ or $G[u_0]$ to find unique functions $l_1$ and $u_1$, such that $l_1$ is the mild solution of 
\begin{align}\label{l1}
\begin{cases}\frac{dl_1}{dt} + v\cdot\nabla_x l_1 = G[l_0] - L[l_1,u_0,u_0,u_0,u_0,u_0]\\
l_1(0) = f_0\end{cases}
\end{align}
and $u_1$ is the mild solution of 
\begin{align}\label{u1}
\begin{cases}
    \frac{du_1}{dt} + v\cdot\nabla_x u_1 = G[u_0] - L[u_1,l_0,l_0,l_0,l_0,l_0]\\
    u_1(0) = f_0.
\end{cases}\end{align}
Under certain conditions on functions $l_0, u_0, l_1, u_1$,  the recursive application of this process results in the following theorem.
\begin{theorem}\label{thm:KS_WP}
     Let $d=1, \alpha,\beta > 0,$ and $C_{1,\beta}$ be defined  in \eqref{conv_constant}.  Consider functions  $f_0 \in \m^+$, $T^{-(\cdot)}l_0, T^{-(\cdot)}u_0 \in L^\infty(\R,\m^+)$  with  $||\T u_0||_{\alpha,\beta} < \frac{\alpha^{1/8}}{\sqrt[4]{240C_{1,\beta}}}$. 
    Let $l_1,u_1$ be the mild solutions of \eqref{l1} and $\eqref{u1}$ respectively and assume the beginning condition 
    \begin{align}\label{beginning}
    0 \leq  T^{-(\cdot)}l_0 \leq T^{-(\cdot)}l_1 \leq T^{-(\cdot)} u_1 \leq  T^{-(\cdot)}u_0
    \end{align}
    holds.
    Then we have: 
    \begin{enumerate}[label = (\roman*)]
        \item There are unique sequences $\{l_n\}_n$ and $\{u_n\}_n$ such that for any $n \in \N$, $l_n,u_n$ are the mild solutions to \eqref{ln} and \eqref{un}, respectively. Furthermore, for any $n \in \N$, we have
        \begin{align}\label{monotonicity}
            0 \leq  T^{-(\cdot)}l_0
            \leq T^{-(\cdot)} l_1
            \leq \dots 
            \leq T^{-(\cdot)} l_n
            \leq T^{-(\cdot)} u_n 
            \leq \dots 
            \leq T^{-(\cdot)} u_1
            \leq  T^{-(\cdot)}u_0.
        \end{align}
        \item For all $t \in \R$, the sequences $\{T^{-t} l_n(t)\}_n$ and $\{T^{-t} u_n(t)\}_n$ converge in $\m^+$ and we denote their limits as follows:
        \begin{align}\label{limits}
            T^{-t}l(t): = \lim\limits_{n\to \infty}T^{-t}l_n(t),\quad T^{-t}u(t): = \lim\limits_{n\to \infty}T^{-t}u_n(t).
        \end{align}
        Furthermore, 
        \[
        T^{-(\cdot)} l, ~T^{-(\cdot)} u \in L^\infty(\R, \m^+) \cap C^0(\R, L^{1+}_{x,v}),
        \]
        and the following integral equations are satisfied, 
        \begin{align}
            T^{-t}l(t) + \int_0^t T^{-\tau} L[l,u,u,u,u,u](\tau)d\tau  = f_0 + \int_0^t T^{-\tau}G[l](\tau)d\tau,\quad \forall~t \in \R \label{l},\\
            T^{-t}u(t) + \int_0^t T^{-\tau} L[u,l,l,l,l,l](\tau)d\tau  = f_0 + \int_0^t T^{-\tau}G[u](\tau)d\tau,\quad \forall~t \in \R.\label{u}
        \end{align}
        \item The limit functions $l,u$ in \eqref{limits} are the same, i.e. $u = l$.
        \item Define $f: = l( = u)$. Then $f$ is the unique mild solution of the 6-wave kinetic equation \eqref{eq:6wave} in the sense of Definition \ref{def:solution} corresponding to initial data $f_0 \in \m^+$ satisfying
        \begin{align}
        \lt \T f\rt \leq \lt \T u_0\rt.
        \end{align}
    \end{enumerate}
\end{theorem}

\begin{proof} We start by proving point (i).
    \begin{enumerate}[label = (\roman*)]
        \item We proceed by induction. 
        Case $n = 1$ follows by assumptions of the theorem.
        Now assume we have constructed $l_1,\dots,l_{n-1},u_1,\dots, u_{n-1}$ satisfying \eqref{ln}-\eqref{un} and \eqref{monotonicity}. 
            Then by  Proposition \ref{prop:ALP} there are unique mild solutions $l_n,u_n$  of \eqref{ln}-\eqref{un}. Thus, it  suffices to show that 
            \begin{align}\label{want}
            \T l_{n-1} \leq \T l_n \leq \T u_n \leq \T u_{n-1}.  
            \end{align}
           The inductive assumption implies  $ l_{n-2} \leq  l_{n-1} \leq  u_{n-1} \leq u_{n-2}$, so by  the monotonicity of $G$ as stated in \eqref{GL_monotonic} we have
            \begin{align*}
                \T G[l_{n-2}] 
                \leq \T G[l_{n-1}] 
                \leq \T G[u_{n-1}] 
                \leq \T G[u_{n-2}].
            \end{align*}
           Now we apply  Corollary \ref{cor:ineq} three times:
           \begin{itemize}
            \item Taking $ g_1 =  u_{n-2},  h_1 =  G[l_{n-2}],  g_2 =  u_{n-1}, h_2 =  G[l_{n-1}]$ implies $\T l_{n-1} \leq \T l_n$.
            \item Taking $g_1 = u_{n-1},    h_1 =  G[l_{n-1}], g_2 =  l_{n-1}, h_2 =  G[u_{n-1}]$ implies $\T l_{n} \leq  \T u_{n}$.
            \item Taking $g_1 =  l_{n-1},  h_1 =  G[u_{n-1}],  g_2 =  l_{n-2},   h_2 =  G[u_{n-2}]$ implies $ \T u_{n} \leq \T u_{n-1}$.
           \end{itemize}
Thus, the inequality  \eqref{want} is satisfied and the claim (i) follows.
      
        \item Thanks to \eqref{monotonicity}, the sequence $\{\T l_n\}_n$ is increasing and upper bounded and the sequence $\{\T u_n\}_n$ is decreasing and lower bounded, so they are convergent. Define 
        \[
        T^{-t} l(t): = \lim\limits_{n\to \infty} T^{-t} l_n(t),\qquad T^{-t} u(t): = \lim\limits_{n\to \infty} T^{-t} u_n(t), \qquad ~\forall t\in \R.
        \]
         Since $ \T u_0 \in L^\infty(\R, \m^+)$,  the inequality \eqref{monotonicity} implies that $\T l,~ \T u \in L^\infty(\R,\m^+)$ and that the above  convergence is in $\m$, i.e.
         \begin{align*}
        &(T^{-t} l_n(t), T^{-t} u_{n-1})(t) \overset{\m}{\longrightarrow} (T^{-t} l(t), T^{-t} u(t)), \quad \forall t\in\R, \\
        & (T^{-t} u_n(t), T^{-t} l_{n-1})(t) \overset{\m}{\longrightarrow}(T^{-t} u(t),T^{-t} l(t)),  \quad \forall t\in\R.
        \end{align*} 
        Thus, Corollary \ref{cor:l1_conv} implies that for any $t \in \R$,  
        \begin{align}
            T^{-t}L[l_n,u_{n-1}, u_{n-1},u_{n-1},u_{n-1},u_{n-1}](t) &\overset{L^1_{x,v}}{\longrightarrow}T^{-t}L[l,u,u,u,u,u](t),\label{conv1}\\
            T^{-t}L[u_n,l_{n-1},l_{n-1},l_{n-1},l_{n-1},l_{n-1}](t) &\overset{L^1_{x,v}}{\longrightarrow}T^{-t}L[u,l,l,l,l,l](t), \label{conv2}\\
            T^{-t}G[l_{n-1}](t) &\overset{L^1_{x,v}}{\longrightarrow} T^{-t}G[l](t), \label{conv3}\\ 
            T^{-t}G[u_{n-1}](t) &\overset{L^1_{x,v}}{\longrightarrow} T^{-t}G[u](t).\label{conv4}
        \end{align}   
       Since $l_n$ is a mild solution of \eqref{ln} and $u_n$ is a mild solution of \eqref{un}, we have for all $t\in\R$:
        \begin{align*}
           &\qquad T^{-t}l_n(t) + \int_0^t T^{-\tau}L[l_n,u_{n-1},u_{n-1},u_{n-1},u_{n-1},u_{n-1}](\tau)d\tau = f_0 + \int_0^t T^{-\tau}G[l_{n-1}](\tau)d\tau,\\
           & \qquad T^{-t}u_n(t) + \int_0^t T^{-\tau}L[u_n,l_{n-1},l_{n-1},l_{n-1}, l_{n-1}, l_{n-1}](\tau)d\tau = f_0 + \int_0^t T^{-\tau}G[u_{n-1}](\tau)d\tau.
        \end{align*}
        Then, by monotonicity of the operators $L$ and $G$, we have  $G[l_{n-1}],G[u_{n-1}] \leq G[u_0]$,   $L[l_n,u_{n-1},u_{n-1},u_{n-1}, u_{n-1}, u_{n-1}] \leq L[u_0]$, and $ L[u_n, l_{n-1},l_{n-1},l_{n-1},l_{n-1},l_{n-1}]\leq L[u_0]$, where $L[u_0] = L[u_0, u_0, u_0, u_0, u_0, u_0]$.
        Additionally, by Proposition \ref{prop:DCT} we have that
         $||T^{-t}L[u_0]||_{L^1_{x,v}}, ||T^{-t}G[u_0]||_{L^1_{x,v}} \in L^1_t$, so
        we may apply the dominated convergence theorem to conclude that by letting $n \to \infty$ and using convergences \eqref{conv1}-\eqref{conv4},  $\T l$ and $\T u$ satisfy \eqref{l} and \eqref{u}. Finally, thanks to \eqref{l} and \eqref{u} we have that  $T^{-(\cdot)} l, ~T^{-(\cdot)} u \in  C^0(\R, L^{1+}_{x,v}).$

        \item Since by \eqref{monotonicity} we have $\T l_0 \le \T l_n \leq \T u_n \le \T u_0$, by letting $n \to \infty$, we have $0 \leq \T l_0 \leq \T l \leq \T u \leq \T u_0$. 
        Combining this  with  \eqref{l}- \eqref{u} and \eqref{G1 decomp}-\eqref{L2 decomp}, we get
        \begin{align*}
            |T^{-t}u(t) - T^{-t}l(t)| 
            & \leq \int_0^t|T^{-\tau}G_1[u](\tau) - T^{-\tau}G_1[l](\tau)| + |T^{-\tau}G_2[u](\tau) - T^{-\tau}G_2[l](\tau)|\\
            &\qquad + |T^{-\tau}L_1[l,u,u,u,u,u](\tau) - T^{-\tau}L_1[u,l,l,l,l,l](\tau)| \\
            &\qquad + |T^{-\tau}L_1[l,u,u,u,u,u](\tau) - T^{-\tau}L_2[u,l,l,l,l,l](\tau)|d\tau\\
            &\leq 120\,C_{1,\beta} \alpha^{-1/2}\ld\T u_0\rd^4\lt \T u - \T l\rt e^{-\alpha|x|^2 - \beta|v|^2}
        \end{align*}
         where in the last inequality we use Proposition \ref{prop:gl_bounds}.
   Since by assumption, $\ld\T u_0\rd < \frac{\alpha^{1/8}}{\sqrt[4]{240 C_{1,\beta}}}$, we have $\lt \T u - \T l\rt  < \frac{1}{2}\lt \T u - \T l\rt$, which implies $u = l$.
    \item Existence follows by defining $f:= l = u \in C(\R, \m^+)\cap L^\infty(\R, \m^+)$. The proof of uniqueness follows from a very similar argument as in part (iii) with $|||T^{-(\cdot)}f|||_{\alpha,\beta} \leq |||T^{-(\cdot)}u_0|||_{\alpha,\beta}$ to get a contraction. 
    \end{enumerate}
\end{proof}

\subsubsection{Proof of Theorem \ref{thm:KSnonnegative}} \label{sec:application of KS}
    Let $f_0 \in \m^+$ with $||f_0||_{\alpha,\beta} < \frac{19}{20}\left(\frac{\alpha^{1/8}}{\sqrt[4]{240 C_{1,\beta}}}\right)$. Define 
    
    \begin{align}
    &T^{-t}l_0(t,x,v) = 0, \\
    &T^{-t} u_0(t,x,v) = C_0e^{-\alpha|x|^2 - \beta|v|^2},
    \end{align}
    with 
    \begin{align}
        C_0 = \frac{20}{19}||f_0||_{\alpha,\beta}.
    \end{align}
     Since $\ld \T u_0\rd = C_0 < \frac{\alpha^{1/8}}{(240 C_{1,\beta})^{1/4}}$, in order to use   Theorem \ref{thm:KS_WP}, it is sufficient to show that the beginning condition \eqref{beginning} is satisfied. 
Thanks to the fact that $\T l_0 =0$, the equations for $l_1$ and $u_1$ in the iteration scheme \eqref{l1}-\eqref{u1} simplify to:
    \begin{align}
        &\begin{cases}
           \displaystyle \frac{d}{dt} T^{-t}l_1(t) + T^{-t}l_1(t)T^{-t}R[u_0](t) = 0\\
            l_1(0) = f_0
        \end{cases}\\
        &\begin{cases}
            \displaystyle \frac{d}{dt} T^{-t}u_1(t) = T^{-t}G[u_0](t)\\
            u_1(0) = f_0,
        \end{cases}
    \end{align}
where $R$ is defined in \eqref{def-R}.
    By integrating these equations, we have 
    \begin{align}
        T^{-t}l_1(t) &= f_0\exp\left(-\int_0^t T^{-\tau}R[u_0](\tau)d\tau\right),\quad t \in \R,\\
        T^{-t}u_1(t) &= f_0 + \int_0^t T^{-\tau}G[u_0](\tau)d\tau, \quad t \in \R. \label{Tu_1}
    \end{align}
    Since $f_0, u_0 \geq 0$ and $\T l_0=0$, these representations imply that 
    \begin{align*}
       0= \T l_0 \leq \T l_1(t)  \leq \T u_1. 
    \end{align*}
    Thus, it remains to show that $\T u_1 \leq \T u_0$. We proceed by using \eqref{Tu_1} and Proposition \ref{prop:gl_bounds}:
    \begin{align*}
        T^{-t}u_1(t) &\leq \ld f_0\rd e^{-\alpha|x|^2 - \beta|v|^2} + 12\,C_{1,\beta}\alpha^{-1/2}\lt \T u_0\rt^5e^{-\alpha|x|^2 - \beta|v|^2}\\
        &\leq \left(\ld f_0\rd + \frac{C_0}{20}\right)e^{-\alpha|x|^2 - \beta|v|^2}.
    \end{align*}
Since $C_0$ satisfies
\begin{align*}
    \left(\ld f_0\rd + \frac{C_0}{20}\right) = C_0,
\end{align*}
we have \[T^{-t} u_1(t,x,v) \leq C_0e^{-\alpha|x|^2 - \beta|v|^2} = T^{-t} u_0(t,x,v).\]
Thus, conditions of Theorem \ref{thm:KS_WP} are satisfied, and then by Theorem \ref{thm:KS_WP} (iv), we obtain existence of a unique mild solution of \eqref{eq:6wave}.
$\square$

\section{Scattering}\label{sec:scattering}

In this section we prove results on the long time behavior of solutions to the 6-wave kinetic equation \eqref{eq:6wave} as stated in Theorem \ref{thm: scattering 1}, Theorem \ref{thm:scattering 2} and Theorem \ref{thm: main scattering}. Such results belong to what is often referred to as scattering theory.

Before we present the proof of Theorem \ref{thm: scattering 1}, we provide heuristic motivation for choosing the map for which we will seek a fixed point. By Duhamel's formula, the mild solution (when it exists) to the 6-wave kinetic equation \eqref{eq:6wave} corresponding to the initial data $f_0$ is given by 
    \begin{equation}\label{def-kwe mild solution equation}
        T^{-t} f(t)= f_0+\int_0^t T^{-s}\C[f](s)\,ds,\quad t \in \R.
    \end{equation}
   Suppose  $f_+ \in B(R_s)$ (see \eqref{B(R)}) is such that $\lim\limits_{t\rightarrow \infty}\ld T^{-t}f(t) - f_{+}\rd = 0$. Then, by letting $t \to \infty$ in $\eqref{def-kwe mild solution equation}$, we would formally obtain
    \begin{align}\label{if limit}
        f_+ = f_0 + \int_0^\infty T^{-s} \C[f](s)~ds
    \end{align}
    in the norm $\|\cdot \|_{\alpha,\beta}$.
    Hence, subtracting \eqref{if limit} from \eqref{def-kwe mild solution equation} yields 
    \begin{align}
        T^{-t}f(t) &= f_+ - \int_t^\infty T^{-s}\C[f](s)~ds.
    \end{align}
    If we denote $T^{-\tau}f(\tau)$ by $g(\tau)$, the above line becomes 
    \begin{align}
        g(t) &= f_+ - \int_t^\infty T^{-s}\C[T^sg](s)~ds.
    \end{align}
This motivates us to  define the map $\mathcal{A}: L^\infty(\R, \m) \to L^\infty(\R, \m)$ as follows
    \begin{align}\label{wke A}
        \mathcal{A}(g)(t) = f_+ - \Lambda_{t,\infty}[g],
    \end{align}
where $\Lambda_{t,\infty}$ is given by \eqref{Lambda}. We note that Proposition \ref{prop:gl_bounds} guarantees that the mapping $\mathcal{A}$ is well defined.

In order to prove Theorem \ref{thm: scattering 1}, for a given $f_+$ we will find $f_0$ by looking for a fixed point $g$ of the mapping $\mathcal{A}$ in an appropriate space. Then, we will be able to choose $f_0: = g(0)$. The second part of the statement pertaining to $f_-$ will be proven similarly with minor modifications that will be addressed in the proof below.

\begin{proof}[Proof of Theorem \ref{thm: scattering 1}]
Let $f_+ \in B(R_s). $
     From the definition \eqref{wke A} of the mapping $\mathcal{A}$, and by Lemma \ref{lem:lambda} with $\rho = 2R_{s}$, we have that for any $g \in \overline{B}(2R_{s})$
    \begin{align*}
        \lt  \A(g) \rti 
        & \le \ld f_+ \rd + \lt \Lambda_{t,\infty}[g] \rti 
        \leq R_{s} + \frac{1}{16}R_{s} \leq 2R_{s}.
    \end{align*} 
 Thus, $\A: \overline{B}(2R_{s}) \to \overline{B}(2R_{s})$.
    
    Now we show that $\mathcal{A}$ is a contraction on $\overline{B}(2R_{s})$. For any $f,g \in \overline{B}(2R_{s})$, by Lemma \ref{lem:lambda} we have
    \begin{align*}
        \lt  \A(f) - \A(g) \rti 
        &\le \lt  \Lambda_{t,\infty}[f] - \Lambda_{t,\infty}[g]  \rt 
         \le \frac{1}{4}\lt f - g\rt.
    \end{align*}
    Thus, $\mathcal{A}$ is a contraction  from $\overline{B}(2R_{s})$ to $\overline{B}(2R_{s})$ and we have a unique fixed point $g \in \overline{B}(2R_{s})$, i.e.
    \begin{align}\label{fixed point}
        g(t) = f_+ - \int_t^\infty T^{-s}\C[T^sg](s)~ds.
    \end{align}
This fixed point $g$ allows us to define
\begin{align}
    &f(t) := T^{t}g(t),\label{fg}\\
    &f_0 := f(0) = g(0).\label{set f_0}
\end{align}

Now we will show that $f$ is a mild solution of the 6-wave kinetic equation \eqref{eq:6wave}. 
By setting $t = 0$ in  \eqref{fixed point} and using \eqref{fg},\eqref{set f_0} we have
\begin{align}\label{f_+}
     f_+ &= f_0 + \int_0^\infty T^{-s}\C[f](s)~ds.
\end{align}    
Substituting \eqref{f_+} into \eqref{fixed point} and using \eqref{fg},  we get
\begin{align} \label{6.13}
        T^{-t}f(t) 
        & = f_0 + \int_0^t T^{-s}\C[f](s)~ds. 
\end{align}
Hence $f$ is a mild solution to the 6-wave kinetic equation \eqref{eq:6wave} with initial data $f_0 \in B(2R_{s})$.

Now we verify \eqref{scatter1_lim}. Going back to   \eqref{fixed point} and employing \eqref{fg}, we have 
    \begin{align}
        T^{-t}f(t) - f_+ = - \int_t^\infty T^{-s}\C[f](s)~ds,
    \end{align}
    and therefore 
\begin{align*}
        \lim\limits_{t \to \infty}\ld T^{-t}f(t) - f_+\rd = \lim\limits_{t \to \infty}\ld \int_t^\infty T^{-s}\C[f](s)~ds\rd = 0,
\end{align*}
thanks to \eqref{int_convergent}. 
This completes the proof of the result for $f_+$. To obtain the result for $f_-$, we instead define the map $\mathcal{A}:L^\infty(\R,\m) \to L^\infty(\R,\m)$ by $\mathcal{A}[g] = f_- + \Lambda_{-\infty,t}[g]$ and proceed with the same argument.
\end{proof}

\begin{proof}[Proof of Theorem \ref{thm:scattering 2}] Let $f_0 \in B(R_{s})$ and let $f$ be the mild solution to the 6-wave kinetic equation \eqref{eq:6wave} corresponding to $f_0$, whose existence is guaranteed by Theorem \ref{thm:wke well posed}. Then  for $t,\tau \in \R$ we have
    \begin{align}
      T^{-t} f(t)&= f_0+\int_0^t T^{-s} \C[f](s)\,ds\label{eq:t_sol},\\
      T^{-\tau} f(\tau)&= f_0+\int_0^\tau T^{-s} \C[f](s)\,ds.\label{eq:tau_sol}
    \end{align}
 Taking the difference of \eqref{eq:t_sol} and \eqref{eq:tau_sol}, we have
    \begin{align}\label{t,tau}
        \ld T^{-t}f(t) - T^{-\tau}f(\tau)\rd =\ld\int_\tau^t T^{-s}\C[f](s)~ds\rd \overset{t,\tau \to\infty}{\longrightarrow} 0,
    \end{align}
     thanks to \eqref{int_convergent}. This implies that  $\{T^{-t_m}f(t_m)\}_{m }$ is Cauchy for any sequence $\{t_m\}$ with $t_m \to \infty$.  Then, since $\m$ is complete, there is a unique limit $f_+ : = \lim_{m \rightarrow \infty} T^{-t_m}f(t_m)\in \m$, 
     which is independent of the choice of the sequence $\{t_m\}$ thanks to \eqref{t,tau}. Now we can use the sequential representation of the limit of a function to conclude 
\[\lim\limits_{t\to\infty}\ld T^{-t}f - f_+\rd= 0.\]
Finally we prove that $f_+ \in B(2R_s)$. This follows by taking the limit  $t \to \infty$ in \eqref{eq:t_sol} to get
    \begin{align}\label{u_+ formula}
        f_+ = f_0 + \int_0^\infty T^{-s}\C[f](s)~ds,
    \end{align}
and by using Lemma \ref{lem:lambda}.

This completes the proof of the result for $f_+$. The result for $f_-$ is proven in the same fashion. 
\end{proof}

Let us introduce the following notation for mappings between initial data and scattering states.

\begin{definition}\label{def_umaps}
Let $d\geq 1, \alpha,\beta > 0$. Let $f_0,f_{\pm} \in \m$. Then we write 
$$ 
\mathcal{U}_{0,\pm}(f_0) = f_{\pm}
$$ 
if there exists a unique mild solution $f$ to $\eqref{eq:6wave}$ corresponding to the initial data $f_0$ such that $\lim\limits_{t\rightarrow \pm\infty}\ld T^{-t}f(t) - f_\pm\rd = 0.$

Analogously, we write  
$$
 \mathcal{U}_{\pm,0}(f_\pm) = f_{0}
$$ if there exists a unique mild solution $f$ to $\eqref{eq:6wave}$ corresponding to the initial data $f_0$ such that $\lim\limits_{t\rightarrow \pm\infty}\ld T^{-t}f(t) - f_\pm\rd = 0.$
\end{definition}
Theorems  \ref{thm: scattering 1} and \ref{thm:scattering 2}  imply that the maps $\U_{0,\pm}$ and $\U_{\pm,0}$ are well-defined on $B(R_s)$:
\begin{align} 
    \U_{\pm,0}:
    \begin{cases} 
     B(R_{s}) \to \U_{\pm,0}(B(R_{s})) \subset B(2R_s), \\
     f_\pm \mapsto f_0 \label{U_+0}
    \end{cases} 
\end{align}
\begin{equation} 
    \U_{0,\pm}:
    \begin{cases}
    B(R_{s}) \to \U_{0,\pm}(B(R_{s})) \subset B(2R_s), \\
     f_0 \mapsto f_\pm.\label{U_0+}
\end{cases}
\end{equation}
We next show   that $\U_{0,\pm}$ and $\U_{\pm,0}$ are injective. 

\begin{theorem}\label{thm:inject of U}
    Let $d = 1,\alpha,\beta > 0$, and $0 < R_{s} \leq \frac{\alpha^{1/8}}{2^{\frac{7}{2}}C_{1,\beta}^{1/4}}$. Then $\U_{0,\pm}$, defined by \eqref{U_0+}, are injective.
\end{theorem}

\begin{proof}
    Let $f_0,g_0 \in B(R_{s})$ be such that $\U_{0,+}(g_0) = \U_{0,+}(f_0)$. Then by \eqref{u_+ formula},
    \begin{align}
        f_0 - g_0 = \Lambda_{0,\infty}[T^{-t}g(t)] - \Lambda_{0,\infty}[T^{-t}f(t)].
    \end{align}
Hence, by Lemma \ref{lem:lambda} with $\rho = R_{s}$ we have
\begin{align*}
    \ld f_0 - g_0\rd  = \lt f_0 - g_0 \rt & = \lt \Lambda_{0,\infty}[T^{-t}g(t)] - \Lambda_{0,\infty}[T^{-t}f(t)]\rt \\
    & \leq \frac{1}{4}\lt T^{-t} f(t) - T^{-t}g(t)\rt\\
    &\leq \frac{1}{2}\ld f_0 - g_0\rd,
\end{align*}
where the last inequality follows from Theorem \ref{thm:wke well posed}.
This implies $f_0 = g_0$ in the sense of the norm $\ld\cdot\rd$, completing the proof of the injectivity of $\mathcal{U}_{0,+}$. One can prove the injectivity of $\mathcal{U}_{0,-}$ analogously.
\end{proof}

\begin{theorem} 
    Let $d = 1,\alpha,\beta > 0$, and $0 < R_{s} \leq \frac{\alpha^{1/8}}{2^{\frac{7}{2}}C_{1,\beta}^{1/4}}$. Then $\U_{\pm,0}$ defined by \eqref{U_+0} are injective.
\end{theorem}
\begin{proof}
    Let $f_+,g_+ \in B(R_s)$.
    
    First, thanks to the condition on $R_{s}$, we observe that the proof of Theorem \ref{thm: scattering 1} tells us that each of the maps 
    \begin{align*}
        \mathcal{A}_{f+}(h) := f_+ - \Lambda_{t,\infty}[h],\\
        \mathcal{A}_{g+}(h) := g_+ - \Lambda_{t,\infty}[h],
    \end{align*}
    has a unique fixed point denoted by $\tilde{f}, \tilde{g} \in \overline{B}(2R_{s})$, respectively. 
    That is, 
    \begin{align*}
       \tilde{f}(t) &= f_+ - \Lambda_{t,\infty}[\tilde{f}],\\
       \tilde{g}(t) &= g_+ - \Lambda_{t,\infty}[\tilde{g}].
    \end{align*}

   Taking the difference, by Lemma \ref{lem:lambda} we have
    \begin{align*}
        \lt \tilde{f} - \tilde{g}\rt &\leq \ld f_+ - g_+\rd + \lt \Lambda_{t,\infty}[\tilde{f}] - \Lambda_{t,\infty}[\tilde{g}]\rt\\
        & \leq \ld f_+ - g_+\rd + \frac{1}{4}\lt \tilde{f} - \tilde{g}\rt,
    \end{align*}
    which implies the following stability estimate: 
    \begin{align}\label{U_+0 stability}
        \lt\tilde{f} - \tilde{g}\rt \leq 2\ld f_+ - g_+\rd.
    \end{align}

    With this estimate in hand, we can proceed as in the proof of Theorem \ref{thm:inject of U}. Assume additionally that $\U_{+,0}(f_+) = \U_{+,0}(g_+)$. 
    By the proof of Theorem \ref{thm: scattering 1}, we also have 
    \begin{align*}
        \U_{+,0}(f_+) &= f_+ - \Lambda_{0,\infty}[\tilde{f}],\\
        \U_{+,0}(g_+) &= g_+ - \Lambda_{0,\infty}[\tilde{g}].
    \end{align*}

   Therefore, using the fact that $\U_{+,0}(g_+) = \U_{+,0}(f_+)$ and applying Lemma \ref{lem:lambda}
   \begin{align*}
       \ld f_+ - g_+\rd &= \lt \Lambda_{0,\infty}[\tilde{g}] - \Lambda_{0,\infty}[\tilde{f}]\rt\\
       & \leq \frac{1}{4}\lt \tilde{f} - \tilde{g}\rt\\
       & \leq \frac{1}{2}\ld f_+ - g_+\rd,
   \end{align*}
   where to obtain the last inequality we use \eqref{U_+0 stability}.
    This implies $f_+ = g_+$ in the sense of the norm $||\cdot||_{\alpha,\beta}$, completing the proof of the injectivity of $\mathcal{U}_{+,0}$. One can prove the injectivity of $\mathcal{U}_{-,0}$ analogously.
\end{proof}

\begin{proof}[Proof of Theorem \ref{thm: main scattering}]
   (i)  Let $f_{-} \in B(\frac{R_s}{2})$. By Theorem \ref{thm: scattering 1}, there exists initial data $f_0 \in B(R_s)$ such that the corresponding solution $f$ of \eqref{eq:6wave} satisfies $\lim\limits_{t \to  -\infty}\ld T^{-t} f(t) - f_-\rd = 0$. Now, by Theorem \ref{thm:scattering 2}, there exists $f_+ \in B(2R_s)$ such that $\lim\limits_{t \to  +\infty}\ld T^{-t}f(t) - f_+\rd = 0$. 

   (ii)  The proof pertaining to the mapping $g_+ \mapsto g_-$ follows identically. \end{proof}

\appendix
\section{Linear decomposition} \label{lin_decomp}

\proof[Proof of \eqref{G1 decomp}]
We begin with the decomposed form the right-hand side and show that it is equivalent to the difference.
\begin{align*} 
    G_1&[f - \widetilde{f},g,h,k,l,m] + G_1[\widetilde{f},g,h,k - \widetilde{k},l,m] + G_1[\widetilde{f},g,h,\widetilde{k},l - \widetilde{l}, m]\nonumber\\
    \qquad &+ G_1[\widetilde{f},g,h,\widetilde{k}, \widetilde{l}, m - \widetilde{m}] + G_1[\widetilde{f},g - \widetilde{g},h - \widetilde{h},\widetilde{k},\widetilde{l},\widetilde{m}]\\
    & = \int_{\R^{5d}}\delta(\Sigma)\delta(\Omega)\bigg[k\,l\,m\,(f - \widetilde{f})(g + h) + (k - \widetilde{k})l\,m\,\widetilde{f}\,(g + h) + \widetilde{k}(l - \widetilde{l})m\,\widetilde{f}(g + h)\\
    &\qquad + \widetilde{k}\,\widetilde{l}(m - \widetilde{m})\widetilde{f}\,(g + h) + \widetilde{k}\,\widetilde{l}\,\widetilde{m}\,\widetilde{f}(g - \widetilde{g} + h - \widetilde{h})\bigg]dv_1dv_2dv_3dv_4dv_5\\
    & = \int_{\R^{5d}}\delta(\Sigma)\delta(\Omega)\bigg[k\,l\,m\,(f - \widetilde{f}) + (k - \widetilde{k})l\,m\,\widetilde{f} + \widetilde{k}(l - \widetilde{l})m\,\widetilde{f} + \widetilde{k}\,\widetilde{l}\,(m - \widetilde{m})\widetilde{f}\bigg](g + h)\\
    &\qquad + \widetilde{k}\,\widetilde{l}\,\widetilde{m}\,\widetilde{f}\,(g - \widetilde{g} + h - \widetilde{h})dv_1dv_2dv_3dv_4dv_5\\
    & = \int_{\R^{5d}}\delta(\Sigma)\delta(\Omega)\bigg[k\,l\,m\,f - \widetilde{k}\,\widetilde{l}\,\widetilde{m}\,\widetilde{f}\bigg](g + h)+ \widetilde{k}\,\widetilde{l}\,\widetilde{m}\,\widetilde{f}(g - \widetilde{g} + h - \widetilde{h})dv_1dv_2dv_3dv_4dv_5\\
    & = \int_{\R^{5d}}\delta(\Sigma)\delta(\Omega)\bigg[k\,l\,m\,f\,(g + h) - \widetilde{k}\,\widetilde{l}\,\widetilde{m}\,\widetilde{f}\,(\widetilde{g} + \widetilde{h})\bigg]dv_1dv_2dv_3dv_4dv_5\\
    & = G_1[f,g,h,k,l,m] - G_1[\widetilde{f},\widetilde{g},\widetilde{h},\widetilde{k},\widetilde{l},\widetilde{m}]
\end{align*}
 The decompositions \eqref{G2 decomp},\eqref{L1 decomp}, and \eqref{L2 decomp} for $G_2, L_1,$ and $L_2$ follow similarly.

\section{Estimates}
\begin{lemma}[\cite{am24}]\label{lem:time}
    Let $x_0,u_0 \in \R^d$ with $u_0 \neq 0$ and $\alpha > 0$. then the following estimate holds \[\int_{-\infty}^\infty e^{-\alpha|x_0 + s u_0|^2}ds \leq \sqrt{\pi}\alpha^{-1/2}|u_0|^{-1}.\]
\end{lemma}

\begin{lemma}\label{lem:conv}
    Let $\beta > 0$, 
     $q \in (-2d,\infty)$. Then the following convolution estimate holds: 
    For any $v \in \R^d$, we have \[\int_{\R^d}(|v_1 - v| + |v_2 - v|)^qe^{-\beta(|v_1|^2 + |v_2|^2)}dv_1dv_2 \leq C_{d,\beta}(1 + |v|^{q+}),\]
    where 
    \begin{align}\label{conv_constant}
        C_{d,\beta} = 2^{q/2}C_d\left(\beta^{-d}  + \frac{1}{2d + q}\right)\1_{q \leq 0} +   2^{3q - 2}C_d\beta^{-q/2 + d}\Gamma\left(\frac{q + d}{2}\right)\1_{q > 0}.
    \end{align}
\end{lemma}
\begin{proof}
    Define $\mathbf{u} = \begin{pmatrix}v_1 - v\\v_2 - v\end{pmatrix}$. First consider the case where $q \in (-2d,0]$.\\ Since $\displaystyle\int_{\R^d}e^{-\beta(|v_1|^2 + |v_2|^2)}dv_1dv_2 \leq C_d\beta^{-d}$, we have
    \begin{align*}
        \int_{\R^{2d}}(|v_1 - v| + |v_2 - v|)^q&e^{-\beta(|v_1|^2 + |v_2|^2)}dv_1dv_2 \\
        &\leq 2^{q/2}\int_{\R^{2d}}(|v_1 - v|^2 + |v_2 - v|^2)^{q/2}e^{-\beta(|v_1|^2 + |v_2|^2)}dv_1dv_2 \\
        &\leq 2^{q/2}\left(\int_{|\mathbf{u}| > 1}e^{-\beta(|v_1|^2 + |v_2|^2)}dv_1dv_2 + \int_{|\mathbf{u}| \leq 1}|\mathbf{u}|^qdv_1dv_2\right)\\
        & \leq 2^{q/2}\left(C_d\beta^{-d} + \int_{|\mathbf{y}|  < 1}|\mathbf{y}|^qd\mathbf{y}\right)\\
        & = 2^{q/2}C_d\left(\beta^{-d} + \int_0^1r^{2d-1 + q}dr\right)\\
        & = 2^{q/2}C_{d}\left(\beta^{-d} + \frac{1}{2d + q}\right).
    \end{align*}
    Now consider $q > 0$. Then 
    \begin{align*}
        (|v - v_1| + |v - v_2|)^q \leq (2|v| + |v_1| + |v_2|)^q \leq 2^{2(q-1)}\left(|v|^q + |v_1|^q + |v_2|^q\right).
    \end{align*}
    Then
    \begin{align*}
        \int_{\R^{2d}}& (|v - v_1| + |v - v_2|)^q e^{-\beta(|v_1|^2 + |v_2|^2)}dv_1dv_2\\
        & \leq 2^{3q-2}\int_{\R^{2d}}(|v|^q + |v_1|^q + |v_2|^q)e^{-\beta(|v_1|^2 + |v_2|^2)}dv_1dv_2\\\\
        & \leq 2^{3q-2}\left[2\left(\int_{\R^{2d}}|v_1|^qe^{-\beta|v_1|^2 }dv_1\right)\left(\int_{\R^d}e^{-\beta|v_2|^2}dv_2\right) + |v|^q\left(\int_{\R^{d}}e^{-\beta|v_1|^2}dv_1\right)^2\right]\\
        & \leq 2^{3q-2}\left[2C_d\beta^{-d/2}\int_{\R^d}|v_1|^qe^{-\beta|v_1|^2}dv_1 + |v|^qC_d^2\beta^{-d}\right]\\
        & \leq2^{3q-2}\left[2C_d\beta^{-d/2}\int_0^\infty r^{q+d-1}e^{-\beta r^2}dr + |v|^qC_d^2\beta^{-d}\right]\\
        & \leq 2^{3q-2}\left[2C_d\beta^{-d/2}\left(\frac{1}{2}\beta^{-(q+d)/2}\Gamma\left(\frac{q + d}{2}\right)\right) + |v|^qC_d^2\beta^{-d}\right]\\
        & \leq 2^{3q-2}\left[C_d\beta^{-(q/2 + d)}\Gamma\left(\frac{q + d}{2}\right) + |v|^qC_d^2\beta^{-d}\right]\\
        & \leq 2^{3q - 2}C_d\beta^{-q/2 + d}\Gamma\left(\frac{q + d}{2}\right)(1 + |v|^q)
    \end{align*}
\end{proof}

\end{document}